\documentclass[11pt,a4paper]{article}

\usepackage{amsmath, amssymb, amsthm}
\usepackage[textwidth=2.8cm]{todonotes}
\usepackage[notcite,notref]{showkeys}
\usepackage{yfonts}
\usepackage{mathrsfs}
\usepackage{tikz,tikz-cd}
\usepackage{color} 

\newtheorem{thm}{Theorem}   
\newtheorem{rem}{Remark}   
\newtheorem{lem}{Lemma}   
\newtheorem{prop}{Proposition}   
\newtheorem{cor}{Corollary}   
   
\newtheorem{OP}{Open Problem}    
\newtheorem{defi}{Definition} 

\title{Nonlinear Operator Ideals Between Metric Spaces and Banach Spaces $$\underline{\mathcal{PART}\  \textrm{I}}$$}
\date{}
\author{MANAF ADNAN SALEH SALEH \\ Mathematisches Institut, Universit\"at Jena \\ Ernst$-$Abbe$-$Platz 2, 07743 Jena, Germany\\ E$-$mail: manaf$_{-}$math@yahoo.com}

\begin{document}

\maketitle

\begin{abstract}

In this paper we present part I of nonlinear operator ideals theory between metric spaces and Banach spaces. Building upon the definition of operator ideal between arbitrary Banach spaces of A. Pietsch we pose three types of nonlinear versions of operator ideals. We introduce several examples of nonlinear ideals and the relationships between them. For every space ideal $\mathsf{A}$ can be generated by a
special nonlinear ideal which consists of those Lipschitz operators admitting a factorization through a Banach space $\mathbf{M}\in\mathsf{A}$. We investigate products and quotients of nonlinear ideals. We devote to constructions three types of new nonlinear ideals from given ones. A ``new'' is a rule defining nonlinear ideals $\mathfrak{A}^{L}_{new}$, $\textswab{A}^{L}_{new}$, and $\textfrak{A}^{L}_{new}$ for every $\mathfrak{A}$, $\textswab{A}^{L}$, and $\textfrak{A}^{L}$ respectively, are called a Lipschitz procedure. Considering the class of all stable objects for a given Lipschitz procedure we obtain nonlinear ideals having special properties. We present the  concept  of  a (strongly) $p-$Banach nonlinear ideal ($0<p<1$) and prove that the nonlinear ideals of Lipschitz nuclear operators, Lipschitz Hilbert operators, products and quotient are strongly $r-$Banach nonlinear ideals ($0<r<1$). 
\end{abstract}

\section{NOTATIONS AND PRELIMINARIES}
We introduce concepts and notations that will be used in this article. The letters $X$, $Y$ and $Z$ will denote pointed metric spaces, i.e. each one has a special point designated by $0$. The letters $E$, $F$ and $G$ will denote Banach spaces. The closed unit ball of a Banach space $E$ is denoted by $B_{E}$. The dual space of $E$ is $E^{*}$. The class of all bounded linear operators between arbitrary Banach spaces will be denoted by $\mathfrak{L}$. The symbols $\mathbb{R}$ and $\mathbb{N}$ stand for the set of all real numbers and the set of all natural numbers, respectively. For a Lipschitz mapping $T$ between metric spaces, $Lip(T)$ denotes its Lipschitz constant. 

Given metric spaces $X$ and $Y$, we assume that all Lipschitz functions from $X$ to $Y$ has a special point $0$. For the special case $Y=\mathbb{R}$, the Banach space of real$-$valued Lipschitz functions defined on $X$ that send $0$ to $0$ with the Lipschitz norm $Lip(\cdot)$ will be denoted by $X^{\#}$. The space $X^{\#}$ is called the Lipschitz dual of $X$. 

In this paper, we write $\mathsf{A}$ be a space ideal see \cite [Sec. 2.1] {P78}. Consider $E_{1}\times E_{2}$ be the Cartesian products of $E_{1}$ and $E_{2}$. Put $J_{1} x_{1}:=(x_{1}, 0)$, $J_{2} x_{2}:=(0, x_{2})$, $Q_{1} (x_{1}, x_{2}):=x_{1}$, and $Q_{2} (x_{1}, x_{2}):=x_{2}$.

 Recall the Lipschitz dual operator $S^{\#}$ from $Y^{\#}$ to $X^{\#}$ of a map $S\in Lip(X,Y)$ is defined by $$\left\langle g, Sx\right\rangle_{(Y^{\#},Y)}=\left\langle S^{\#}g,x\right\rangle_{(X^{\#},X)}$$ for every $x\in X$ and $g\in Y^{\#}$. This is a bounded linear operator and $Lip(S)=\left\|S^{\#}\right\|_{\mathfrak{L}(Y^{\#},X^{\#})}$, see \cite{Mass15}. The definition of the extension and lifting property of Banach space can be founded in \cite [Sec. C.3] {P78}. For a Banach space $F$. Then we put $F^{inj}:=\ell_{\infty}(B_{F^{*}})$ and $J_{F} e:=(\left\langle e, e^{*}\right\rangle)$ for $e\in F$. Clearly $J_{F}$ is a metric injection from $F$ into $F^{inj}$. In this way every Banach space can be identified with a subspace of some Banach space having the metric extension property. For a Banach space $E$. Then we put $E^{sur}:=\ell_{1}(B_{E})$ and $Q_{E} (\xi_{x}):=\sum\limits_{B_{E}} \xi_{x} x$ for $(\xi_{x})\in \ell_{1}(B_{E})$. Clearly $Q_{E}$ is a metric surjection from $E^{sur}$ onto $E$. In this way every Banach space can be identified with a quotient space of some Banach space having the metric lifting property. 
     
\section{DEFINITIONS AND ELEMENTARY PROPERTIES}

Building upon the results of M. G. Cabrera--Padilla, J. A. Ch{\'a}vez-Dom{\'{\i}}nguez, A. Jimenez--Vargas and M. Villegas--Vallecillos \cite{MJAM14} we introduce Lipschitz tensor product between pointed metric space $X$ and the topological dual space $F^{*}$ of a Banach space $F$.

Similarly to \cite [Definition 1.1] {MJAM14} we can constructed $X\boxtimes F^{*}$ as a space of linear functionals on $Lip(X,F)$ spanned by the set $\left\{\delta_{(x,y)}\boxtimes a : (x,y)\in X\times X; a\in F^{*}\right\}$.

The proof of the next lemma is similar to \cite [Lemma 1.1] {MJAM14} and is therefore omitted.

\begin{lem}
Let $\lambda\in\mathbb{K}$, $(x,y), (x_{1},y_{1}), (x_{2},y_{2})\in X\times X$ and $a, a_{1}, a_{2}\in F^{*}$.
\begin{itemize}

	\item $\lambda\left(\delta_{(x,y)}\boxtimes a\right)=\left((\lambda\delta_{(x,y)})\boxtimes a\right)=\left(\delta_{(x,y)}\boxtimes (\lambda a)\right)$.
	
	\item $\left(\delta_{(x_{1},y_{1})}+\delta_{(x_{2},y_{2})}\right)\boxtimes a=\delta_{(x_{1},y_{1})}\boxtimes a + \delta_{(x_{1},y_{1})}\boxtimes a$.
	
	\item $\delta_{(x,y)}\boxtimes \left(a_{1}+ a_{2}\right)=\delta_{(x,y)}\boxtimes a_{1} + \delta_{(x,y)}\boxtimes a_{2}$.

	\item $\left(\delta_{(x,y)}\boxtimes a\right)=\delta_{(x,y)}\boxtimes \theta=\theta$.
	
\end{itemize}

\end{lem}

\begin{rem}
We say that $\delta_{(x,y)}\boxtimes a$ is an elementary Lipschitz tensor element in $X\boxtimes F^{*}$. Note that each element $v$ in $X\boxtimes F^{*}$ is of the form $v=\sum\limits_{j=1}^{m}\lambda_{j}\left(\delta_{(x_{j},y_{j})}\boxtimes a_{j}\right)$, where $n\in\mathbb{K}$, $\lambda_{j}\in\mathbb{K}$, $(x_{j},y_{j})\in X\times X$ and $a_{j}\in F^{*}$. This representation of $v$ is not unique. It is worth noting that each element $v$ of $X\boxtimes F^{*}$  can be represented as $v=\sum\limits_{j=1}^{m}\left(\delta_{(x_{j},y_{j})}\boxtimes a_{j}\right)$ since $\lambda\left(\delta_{(x,y)}\boxtimes a\right)=\left(\delta_{(x,y)}\boxtimes (\lambda a)\right)$. This representation of $v$ admits the following refinement.
\end{rem}

The proof of the next lemma is similar to \cite [Lemma 1.1] {MJAM14} and is therefore omitted.

\begin{lem}
Every nonzero Lipschitz tensor $v\in X\boxtimes F^{*}$  has a representation in the form \\ $\sum\limits_{j=1}^{m}\left(\delta_{(z_{j},0)}\boxtimes q_{j}\right)$, where

$$m=\min\left\{k\in\mathbb{N}: \exists z_{1},\cdots, z_{k}, q_{1},\cdots, q_{k}\in F^{*}| v=\sum\limits_{j=1}^{k}\left(\delta_{(z_{k},0)}\boxtimes q_{j}\right)\right\}$$

and the points $z_{1},\cdots, z_{m}$ in $X$ are distinct from the base point $0$ of $X$ and pairwise distinct.

\end{lem}

We can concatenate the representations of two elements of $X\boxtimes F^{*}$ to get a representation of their sum.

\begin{lem}
Let $v_{1},\; v_{2}\in X\boxtimes F^{*}$ and let $\sum\limits_{j=1}^{m}\left(\delta_{(x_{j},y_{j})}\boxtimes a_{j}\right)$ and $\sum\limits_{j=1}^{m'}\left(\delta_{(x'_{j},y'_{j})}\boxtimes a'_{j}\right)$ be representations of $v_{1}$ and $v_{2}$, respectively. Then $\sum\limits_{j=1}^{m+m'}\left(\delta_{(x''_{j},y''_{j})}\boxtimes a''_{j}\right)$, where
\begin{equation}\nonumber 
\left(x''_{j},y''_{j}\right)=\left\{
\begin{array}{rl}
(x_{j},y_{j})   & \text{if } j=1,\cdots,    m    \\
(x'_{j-m},y'_{j-m}) & \text{if } j=m+1,\cdots, m+m' \\
\end{array} \right.
\end{equation}
\begin{equation}\nonumber 
a''_{j}=\left\{
\begin{array}{rl}
a_{j}   & \text{if } j=1,\cdots,    m    \\
a'_{j-m} & \text{if } j=m+1,\cdots, m+m' \\
\end{array} \right.
\end{equation}

is a representation of $v_{1}+ v_{2}$.

\end{lem}

We now describe the action of a Lipschitz tensor $v\in X\boxtimes F^{*}$ on a function $f\in Lip(X,F)$.

The proof of the next lemma is similar to \cite [Lemma 1.4] {MJAM14} and is therefore omitted.

\begin{lem}\label{Auto}
Let $v=\sum\limits_{j=1}^{m}\left(\delta_{(x_{j},y_{j})}\boxtimes a_{j}\right)\in X\boxtimes F^{*}$ and $f\in Lip(X,F)$. Then

$$v(f)=\sum\limits_{j=1}^{m}\left\langle fx_{j}-fy_{j}, a_{j}\right\rangle.$$ 

\end{lem}

Our next aim is to characterize the zero Lipschitz tensor. For it we need the following Lipschitz operators.

The proof of the next lemma is similar to \cite [Lemma 1.5] {MJAM14} and is therefore omitted.

\begin{lem}\label{Auto1}
Let $g\in X^{\#}$ and $e\in F$. The map $g\cdot e :X\longrightarrow F$, given by
$$\left(g\cdot e\right)(x)=g(x)\cdot e,$$ belongs to $Lip(X,F)$ and $Lip\left(g\cdot e\right)=Lip(g)\cdot\left\|e\right\|$. 
\end{lem}

The proof of the next proposition is similar to \cite [Proposition 1.6] {MJAM14} and is therefore omitted.

\begin{prop}
If $v=\sum\limits_{j=1}^{m}\left(\delta_{(x_{j},y_{j})}\boxtimes a_{j}\right)\in X\boxtimes F^{*}$, then the following assertions are equivalent:
\begin{itemize}
	\item $v=0$.
	\item $\sum\limits_{j=1}^{m} \left(gx_{j}-gy_{j}\right)\cdot\left\langle a_{j}, e\right\rangle=0$, for every $g\in B_{{X}^{\#}}$ and $e\in B_{E}$.
	\item $\sum\limits_{j=1}^{m} \left(gx_{j}-gy_{j}\right)\cdot a_{j}=0$, for every $g\in B_{{X}^{\#}}$.
\end{itemize}
\end{prop}

The proof of the next theorem is similar to \cite [Theorem 1.7] {MJAM14} and is therefore omitted.

\begin{thm}
$\left\langle X\boxtimes F^{*}, Lip(X,F)\right\rangle$ forms a dual pair, where the bilinear form $\left\langle \cdot, \cdot\right\rangle$associated to the dual pair is given, for $v=\sum\limits_{j=1}^{m}\left(\delta_{(x_{j},y_{j})}\boxtimes a_{j}\right)\in X\boxtimes F^{*}$ and 
$f\in Lip(X,F)$, by $$\left\langle v, f\right\rangle=\sum\limits_{j=1}^{m}\left\langle fx_{j}-fy_{j}, a_{j}\right\rangle.$$
\end{thm}

Since $\left\langle X\boxtimes F^{*}, Lip(X,F)\right\rangle$ is a dual pair, $Lip(X,F)$ can be identified with a linear subspace of $\left( X\boxtimes F^{*}\right)'$ as follows.

The proof of the next corollary is similar to \cite [Corollary 1.8] {MJAM14} and is therefore omitted.

\begin{cor}
For every map $f\in Lip(X,F)$, the functional $\Lambda(f):X\boxtimes F^{*}\longrightarrow\mathbb{K}$, given by

$$\Lambda(f)(v)=\sum\limits_{j=1}^{m}\left\langle fx_{j}-fy_{j}, a_{j}\right\rangle$$
for $v=\sum\limits_{j=1}^{m}\left(\delta_{(x_{j},y_{j})}\boxtimes a_{j}\right)\in X\boxtimes F^{*}$, is linear. We say $\Lambda(f)$ is the linear functional on $X\boxtimes F^{*}$ associated to $f$. The map $f\mapsto \Lambda(f)$ is a linear monomorphism from  $Lip(X,F)$ into $\left( X\boxtimes F^{*}\right)'$. 
\end{cor}

Similarly to \cite [Definition 2.1] {MJAM14} we introduce the concept of Lipschitz tensor product functional of a Lipschitz functional and a bounded linear functional as follows.

\begin{defi}
Let $X$ be a pointed metric space and $F$ a Banach space. Let $g\in X^{\#}$ and $a_{j}\in F^{*}$. The map 
$g\boxtimes e : X\boxtimes F^{*}\longrightarrow\mathbb{K}$, given by

$$\left(g\boxtimes e\right)(v)=\sum\limits_{j=1}^{m} \left(gx_{j}-gy_{j}\right)\cdot\left\langle a_{j}, e\right\rangle$$

for $v=\sum\limits_{j=1}^{m}\left(\delta_{(x_{j},y_{j})}\boxtimes a_{j}\right)\in X\boxtimes F^{*}$, is called the Lipschitz tensor product functional of $g$ and $e$.

By Lemma \ref{Auto} note that

$$\left(g\boxtimes e\right)(v)=\sum\limits_{j=1}^{m} \left((g\cdot e)x_{j}-(g\cdot e)y_{j}\right)\cdot\left\langle a_{j}, e\right\rangle=v(g\cdot e).$$
\end{defi}

Similarly to \cite [Lemma 2.1] {MJAM14}, the following result which follows easily from this formula gathers some properties of these functionals.

\begin{lem}
Let $g\in X^{\#}$ and $e\in F$. The functional $\left(g\boxtimes e\right): X\boxtimes F^{*}\longrightarrow\mathbb{K}$ is a well--defined linear map satisfying $\lambda\left(g\boxtimes e\right)=\left(\lambda\cdot g\right)\boxtimes e=g\boxtimes \left(\lambda\cdot e\right)$ for any $\lambda\in\mathbb{K}$. Moreover $\left(g_{1}+g_{2}\right)\boxtimes e=g_{1}\boxtimes e + g_{2}\boxtimes e$ for all $g_{1}, g_{2}\in X^{\#}$ and $g\boxtimes \left(e_{1} + e_{2} \right)=g\boxtimes e_{1} + g\boxtimes e_{2}$ for all $e_{1},\; e_{2}\in F$.
\end{lem}

Similarly to \cite [Definition 2.2] {MJAM14} we introduce the concept of associated Lipschitz tensor product space of $X\boxtimes F^{*}$.

\begin{defi}
Let $X$ be a pointed metric space and $F$ a Banach space. The space $X^{\#}\boxplus F$ is defined as the linear subspace of $\left(X\boxtimes F^{*}\right)'$ spanned by the set $\left\{g\boxtimes e : g\in X^{\#},\; e\in F\right\}$. This space is called the associated Lipschitz tensor product of $X\boxtimes F^{*}$. 
\end{defi}

Similarly to \cite [Lemma 2.2] {MJAM14} we also derive easily the following fact.

\begin{lem}
For any $\sum\limits_{j=1}^{m} g_{j}\boxtimes e_{j}\in X^{\#}\boxplus F$ and $\sum\limits_{j=1}^{m}\delta_{(x_{j},y_{j})}\boxtimes a_{j}\in X\boxtimes F^{*}$, we have

$$\left(\sum\limits_{j=1}^{m} g_{j}\boxtimes e_{j}\right)\left(\sum\limits_{j=1}^{m}\delta_{(x_{j},y_{j})}\boxtimes a_{j}\right)=\left(\sum\limits_{j=1}^{m}\delta_{(x_{j},y_{j})}\boxtimes a_{j}\right)\left(\sum\limits_{j=1}^{m} g_{j}\cdot e_{j}\right).$$

\end{lem}

Each element $v^{*}$ in  $X^{\#}\boxplus F$ has the form $\sum\limits_{j=1}^{m} \lambda_{j} \left(g_{j}\boxtimes e_{j}\right)$, where $m\in\mathbb{N}$, $\lambda_{j}\in\mathbb{K}$, $g_{j}\in X^{\#}$ and $e_{j}\in F$  but this representation is not unique. Since $\lambda\left(g\boxtimes e\right)=\left(\lambda g\right)\boxtimes e= g\boxtimes\left(\lambda e\right)$, each element of $X^{\#}\boxplus F$ can be expressed
as $\sum\limits_{j=1}^{m} g_{j}\boxtimes e_{j}$. This representation can be improved as follows.

The proof of the next lemma is similar to \cite [Lemma 2.3] {MJAM14} and is therefore omitted.

\begin{lem}
Every nonzero element $v^{*}$ in $X^{\#}\boxplus F$ has a representation $\sum\limits_{j=1}^{m} g_{j}\boxtimes e_{j}$ such that the functions
$g_{1},\cdots, g_{m}$ in $X^{\#}$ are nonzero and the elements $e_{1},\cdots, e_{m}$ in $F$ are linearly independent.
\end{lem}

For a pointed metric space $X$ and a Banach space $E$, we denote by $\textfrak{F}^{L}(X,F)$ the set of all Lipschitz finite rank operators from $X$ to $F$. Clearly, $Lip_{F}(X,F)$ is a linear subspace of $Lip(X,F)$. For any $g\in X^{\#}$ and $e\in F$, we consider in Lemma \ref{Auto1} the elements $g\cdot e$ of $Lip_{F}(X,F)$ defined by $(g\cdot e)(x) = g(x)\cdot e$ for all $x\in X$. Note that
rank$(g\cdot e) = 1$ if $g\neq 0$ and $e\neq 0$. Now we prove that these elements generate linearly the space $Lip_{F}(X,F)$.

The proof of the next lemma is similar to \cite [Lemma 2.4] {MJAM14} and is therefore omitted.

\begin{lem}
Every element $T\in Lip_{F}(X,F)$ has a representation in the form 
\begin{equation}\label{Auto2}
T=\sum\limits_{j=1}^{m} g_{j}\cdot e_{j},
\end{equation}
where $m=rank(f)$, $g_{1},\cdots, g_{m}$ in $X^{\#}$ and $e_{1},\cdots, e_{m}$ in $F$.
\end{lem}

The proof of the next theorem is similar to \cite [Theorem 2.5] {MJAM14} and is therefore omitted.

\begin{lem}\label{Auto11}
Let $v^{*}=\sum\limits_{j=1}^{m} g_{j}\boxtimes e_{j}$ in $X^{\#}\boxplus F$. The map $\widetilde{v}$ from $X$ into $F$, defined by 
$$\widetilde{v}(x)=\sum\limits_{j=1}^{m} g_{j}(x)\cdot e_{j}\stackrel{def}{=} v^{*}(x),$$
is a linear isomorphism.
\end{lem}

\begin{rem}\label{Auto9}
Let $g\in X^{\#}$ and $e\in F$. Obviously $$g\boxtimes e : x \longrightarrow \left\langle x, g\right\rangle\cdot e$$
is a Lipschitz finite rank operator. Also we have $Lip(g\boxtimes e)= Lip(g)\cdot\left\|e\right\|$.
\end{rem}

Recall that the definition of operator ideal between arbitrary Banach spaces of A. Pietsch \cite {P07} and \cite{P87} is as follows. Suppose that, for every pair of Banach spaces $E$ and $F$, we are given a subset $\mathfrak{A}(E,F)$ of $\mathfrak{L}(E,F)$. The class  
$$\mathfrak{A}:=\bigcup_{E,F}\mathfrak{A}(E,F),$$
is said to be an \textbf{operator ideal}, or just an \textbf{ideal}, if the following conditions are satisfied:

\begin{description}\label{Auto5}
	\item[$\bf (OI_0)$] $a^{*}\otimes e\in\mathfrak{A}(E,F)$ for $a^{*}\in E^{*}$ and $e\in F$.
	\item[$\bf (OI_1)$] $S + T\in\mathfrak{A}(E,F)$ for $S,\: T\in\mathfrak{A}(E,F)$.
	\item[$\bf (OI_2)$] $BTA\in\mathfrak{A}(E_{0},F_{0})$ for $A\in \mathfrak{L}(E_{0},E)$, $T\in\mathfrak{A}(E,F)$, and $B\in\mathfrak{L}(F,F_{0})$.
\end{description}

Condition $\bf (OI_0)$ implies that $\mathfrak{A}$ contains nonzero operators. \\ Building upon the linear version of aforementioned concept ''operator ideals'' we present three types of nonlinear versions of operator ideals as follow. 

The fist type is called nonlinear ideal between arbitrary Banach spaces $E$ and $F$ as follow.

\begin{defi}
Suppose that, for every pair of Banach spaces $E$ and $F$, we are given a subset $\textswab{A}^{L}(E,F)$ of $Lip(E,F)$. The class 
 $$\textswab{A}^{L}:=\bigcup_{E,F}\textswab{A}^{L}(E,F),$$
is said to be a \textbf{nonlinear operator ideal}, or just a \textbf{nonlinear ideal}, if the following conditions are satisfied:

\begin{description}\label{Auto6}
	\item[$\bf (NOI_0)$] $h\boxtimes e\in\textswab{A}^{L}(E,F)$ for $h\in E^{\#}$ and $e\in F$.
	\item[$\bf (NOI_1)$] $S + T\in\textswab{A}^{L}(E,F)$ for $S,\: T\in\textswab{A}^{L}(E,F)$.
	\item[$\bf (NOI_2)$] $BTA\in\textswab{A}^{L}(E_{0},F_{0})$ for $A\in Lip(E_{0},E)$, $T\in\textswab{A}^{L}(E,F)$, and $B\in\mathfrak{L}(F,F_{0})$.
\end{description}
\end{defi}

Condition $\bf (NOI_0)$ implies that $\textswab{A}^{L}$ contains nonzero Lipschitz operators. 

\begin{rem}
Suppose that, for every pair of metric spaces $X$ and Banach spaces $F$, the class
$$Lip:=\bigcup_{X,F} Lip(X,F),$$ stands for all Lipschitz maps acting between arbitrary metric spaces and Banach spaces.
\end{rem}

The second type is called nonlinear ideal between arbitrary metric spaces $X$ and arbitrary Banach spaces $F$ as follow.
\begin{defi}\label{Auto7}
Suppose that, for every pair of metric spaces $X$ and Banach spaces $F$, we are given a subset $\textfrak{A}^{L}(X,F)$ of $Lip(X,F)$. The class  
$$\textfrak{A}^{L}:=\bigcup_{X,F}\textfrak{A}^{L}(X,F),$$
is said to be a \textbf{nonlinear operator ideal}, or just a \textbf{nonlinear ideal}, if the following conditions are satisfied:

\begin{description}
	\item[$\bf (\widetilde{NOI_0})$] $g\boxtimes e\in\textfrak{A}^{L}(X,F)$ for $g\in X^{\#}$ and $e\in F$.
	\item[$\bf (\widetilde{NOI_1})$] $S + T\in\textfrak{A}^{L}(X,F)$ for $S,\: T\in\textfrak{A}^{L}(X,F)$.
	\item[$\bf (\widetilde{NOI_2})$] $BTA\in\textfrak{A}^{L}(X_{0},F_{0})$ for $A\in Lip(X_{0},X)$, $T\in\textfrak{A}^{L}(X,F)$, and $B\in\mathfrak{L}(F,F_{0})$.
\end{description}

\end{defi}

Condition $\bf (\widetilde{NOI_0})$ implies that $\textfrak{A}^{L}$ contains nonzero Lipschitz operators.

\begin{rem}
Suppose that, for every pair of metric spaces $X$ and $Y$, the class
$$\mathscr{Lip}:=\bigcup_{X,Y} \mathscr{L}(X,Y),$$ stands for all Lipschitz maps acting between arbitrary metric spaces $X$ and $Y$.
\end{rem}

The third type is called nonlinear ideal between arbitrary metric spaces $X$ and $Y$ as follow.
\begin{defi}\label{Auto7}
Suppose that, for every pair of metric spaces $X$ and $Y$, we are given a subset $\mathscr{A}^{L}(X,Y)$ of $\mathscr{Lip}(X,Y)$. The class  
$$\mathscr{A}^{L}:=\bigcup_{X,Y}\mathscr{A}^{L}(X,Y),$$
is said to be a \textbf{nonlinear operator ideal}, or just a \textbf{nonlinear ideal}, if the following conditions are satisfied:

\begin{description}
	\item[$\bf (\widetilde{\widetilde{NOI_0\,}})$] $\textfrak{A}^{L}\subset\mathscr{A}^{L}$.
	\item[$\bf (\widetilde{\widetilde{NOI_1\,}})$] $BTA\in\mathscr{A}^{L}(X_{0}, Y_{0})$ for $A\in \mathscr{L}(X_{0},X)$, $T\in\mathscr{A}^{L}(X, Y)$, and $B\in\mathscr{L}(Y, Y_{0})$.
\end{description}

\end{defi}

Mathematical interpretation of condition $\bf (\widetilde{\widetilde{NOI_0\,}})$ that the linear space $\textfrak{A}^{L}(X,F)$ define as follow. 
$$\textfrak{A}^{L}(X,F)=\left\{T\in Lip(X,F) : T\in\mathscr{A}^{L}(X,F)\right\},$$ which implies that $\mathscr{A}^{L}$ contains nonzero Lipschitz operators.

\begin{rem}
\begin{itemize}

	\item It is obvious that the so--called components $\textfrak{A}^{L}(X,F)$ of a given nonlinear ideal $\textfrak{A}^{L}$ are linear subspaces of the corresponding $Lip(X,F)$ and that $X^{\#}\boxplus F$ contained in $\textfrak{A}^{L}(X,F)$.

 \item In the Definition \ref{Auto6} and Definition \ref{Auto5}, for every pair of Banach spaces $E$ and $F$ if we define the linear spaces $\mathfrak{A}(E,F)$, $\textswab{A}^{L}(E,F)$, and $\textfrak{A}^{L}(E,F)$, respectively as follow:
$$\textfrak{A}^{L}(E,F)=\left\{T\in Lip(E,F) : T\in\mathscr{A}^{L}(E,F)\right\},$$
$$\textswab{A}^{L}(E,F)=\left\{T\in Lip(E,F) : T\in\textfrak{A}^{L}(E,F)\right\},$$
$$\mathfrak{A}(E,F)=\left\{T\in\mathfrak{L}(E,F) : T\in\textswab{A}^{L}(E,F)\right\},$$ then $\mathfrak{A}\subset\textswab{A}^{L}\subset\textfrak{A}^{L}\subseteq\mathscr{A}^{L}$.

\item Two nonlinear ideals $\textfrak{A}^{L}$ $\textfrak{B}^{L}$ are called equal, $\textfrak{A}^{L}=\textfrak{B}^{L}$, if they coincide component--wise. More generally, we shall frequently write $\textfrak{A}^{L}\subset\textfrak{B}^{L}$ to mean that $\textfrak{A}^{L}(X,F)\subset\textfrak{B}^{L}(X,F)$ holds for every pair of metric spaces $X$ and Banach spaces $F$.   

\end{itemize}

\end{rem}

\begin{prop}\label{Auto4}
Let $\textfrak{A}^{L}$ be a nonlinear ideal. Then all components $\textfrak{A}^{L}(X,F)$ are linear spaces.
\end{prop}

\begin{proof}
By $\bf (\widetilde{NOI_1})$ it remains to show that $T\in\textfrak{A}^{L}(X,F)$ and $\lambda\in\mathbb{K}$ imply $\lambda\cdot T\in\textfrak{A}^{L}(X,F)$. This follows from $\lambda\cdot T=\left(\lambda\cdot I_{F}\right)\circ T\circ I_{X}$ and $\bf (\widetilde{NOI_2})$. \\
\end{proof}

\begin{defi}
A nonlinear ideal $\textfrak{A}^{L}$ is said to be closed if all components $\textfrak{A}^{L}(X,F)$ are closed in $Lip(X,F)$ with respect to the Lipschitz operator norm.
\end{defi}

\begin{rem}
The adjectives '' largest'' and ''smallest'' used below in connection with nonlinear ideals always refer to this '' order relation''.
\end{rem}
Certain important nonlinear ideals of Lipschitz operators present as follow. 

\subsection{Largest Nonlinear Ideals}

The class $Lip:=\bigcup_{X,F} Lip(X,F)$ is the largest nonlinear ideal obviously obtained by considering all Lipschitz operators between metric spaces and Banach spaces.


\subsection{Lipschitz Finite Operators (Smallest Nonlinear Ideal)}
We write formula (\ref{Auto2}) in the language of Lipschitz tensor products. A Lipschitz operator $T\in Lip(X,F)$ is finite if it can be written in the form

\begin{equation}\label{Auto3}
T=\sum\limits_{j=1}^{m} g_{j}\boxtimes e_{j},
\end{equation}
where $g_{1},\cdots, g_{m}$ in $X^{\#}$ and $e_{1},\cdots, e_{m}$ in $F$. 

The class of all Lipschitz finite rank operators is denoted by $\textfrak{F}^{L}$.

\begin{lem}
$\textfrak{F}^{L}$ is a smallest nonlinear ideal. 
\end{lem}

\begin{proof}
By Lemma \ref{Auto11}, we have $X^{\#}\boxplus F=\textfrak{F}^{L}(X,F)$ for all metric spaces $X$ and Banach spaces $F$. Since $X^{\#}\boxplus F\subset\textfrak{A}^{L}(X,F)$ holds for every nonlinear ideal $\textfrak{A}^{L}$. This proves $\textfrak{F}^{L}\subseteq\textfrak{A}^{L}$. The verification of the nonlinear ideal properties is following. 

The condition $\bf (\widetilde{NOI_0})$ is hold. To prove the condition $\bf (\widetilde{NOI_1})$, let $T$ and $S$ in $\textfrak{F}^{L}(X,F)$. Then $\sum\limits_{j=1}^{m} g_{j}\cdot e_{j}$ and $\sum\limits_{j=1}^{n} g'_{j}\cdot e'_{j}$ be representation of $T$ and $S$, respectively. Then $\sum\limits_{j=1}^{m + n} g''_{j}\cdot e''_{j}$, where
\begin{equation}\nonumber 
g''_{j}=\left\{
\begin{array}{rl}
g _{j}       & \text{if } j=1,\cdots,    m    \\
g'_{j-m}     & \text{if } j=m + 1,\cdots, m + n \\
\end{array} \right.
\end{equation}
\begin{equation}\nonumber 
e''_{j}=\left\{
\begin{array}{rl}
e_{j}       & \text{if } j=1,\cdots,    m    \\
e'_{j-m}    & \text{if } j=m + 1,\cdots, m + n \\
\end{array} \right.
\end{equation}

is a representation of $T + S$. Since by the definition, both $g : X \longrightarrow \mathbb{K}$ and $A : X_{0} \longrightarrow X$ are Lipschitz maps with $A (0)=0$, then $g\circ A : X_{0}\longrightarrow \mathbb{K}$ is a Lipschitz map and $g\circ A\in X_{0}^{\#}$. Let $h=g\circ A$, $A\in Lip(X_{0},X)$, $T\in\textfrak{F}^{L}(X,F)$, and $B\in\mathfrak{L}(F,F_{0})$, to show that $BTA\in\textfrak{F}^{L}(X_{0},F_{0})$. 
\begin{align}
\left(BTA\right)(x_{0})=\left[B \left(\sum\limits_{j=1}^{m} g_{j}\cdot e_{j}\right) A\right](x_{0})&=\left[B \left(\sum\limits_{j=1}^{m} g_{j}\cdot e_{j}\right)\right] A(x_{0})=B\left[\left(\sum\limits_{j=1}^{m} g_{j}\cdot e_{j}\right) A(x_{0})\right]   \nonumber \\
&=B\left[\left(\sum\limits_{j=1}^{m} g_{j}(A x_{0})\cdot e_{j}\right)\right]=\sum\limits_{j=1}^{m} g_{j}(A x_{0})\cdot Be_{j}   \nonumber \\
&=\sum\limits_{j=1}^{m} \left(g_{j} \circ A\right) x_{0}\cdot Be_{j}=\sum\limits_{j=1}^{m} h_{j} (x_{0})\cdot Be_{j} \nonumber \\
&=\sum\limits_{j=1}^{m} \left(h_{j}\cdot Be_{j}\right)(x_{0}). \nonumber
\end{align}
Hence $\sum\limits_{j=1}^{m} h_{j}\cdot Be_{j}$ be a representation of $BTA$, the condition $\bf (\widetilde{NOI_2})$ satisfied. \\
\end{proof}

\subsection{Lipschitz Approximable Operators}\label{Auto 15}
Recall that the definition of Lipschitz approximable operators of A. Jim{\'e}nez--Vargas, J. M. Sepulcre, and Mois{\'e}s Villegas--Vallecillos, \cite{JAJM14} is as follows. A Lipschitz operator $T\in Lip(X,F)$ is called Lipschitz approximable if there are $T_{1}, T_{2}, T_{3},\cdots\in\textfrak{F}^{L}(X, F)$ with $\lim\limits_{n} Lip\left(T - T_{n}\right)=0$. 

The class of all Lipschitz approximable operators is denoted by $\textfrak{G}^{L}$.

\begin{lem}
$\textfrak{G}^{L}$ is a nonlinear ideal. 
\end{lem}

\begin{proof}
The Lipschitz tensor product $g\boxtimes e\in\textfrak{G}^{L}(X, F)$. Since there is a sequence $\left(T_{n}\right)_{n\in\mathbb{N}}=g\boxtimes e$ in $\textfrak{F}^{L}(X, F)$ with $\lim\limits_{n} Lip\left(g\boxtimes e - T_{n}\right)=0$. To prove the condition $\bf (\widetilde{NOI_1})$, let $T$ and $S$ in $\textfrak{G}^{L}(X,F)$. Then there are sequences $\left(T_{n}\right)_{n\in\mathbb{N}}$ and $\left(S_{n}\right)_{n\in\mathbb{N}}$ in $\textfrak{F}^{L}(X, F)$ with $\lim\limits_{n} Lip\left(T - T_{n}\right)=\lim\limits_{n} Lip\left(S - S_{n}\right)=0$. From Proposition \ref{Auto4} there is a sequence $\left((T_{n} + S_{n})\right)_{n\in\mathbb{N}}$ in $\textfrak{F}^{L}(X, F)$ such that $\lim\limits_{n} Lip\left(T + S - (T_{n} + S_{n})\right)=0$.

Let $A\in Lip(X_{0},X)$, $T\in\textfrak{G}^{L}(X, F)$, and $B\in\mathfrak{L}(F,F_{0})$. Then there is a sequence $\left(T_{n}\right)_{n\in\mathbb{N}}$ in $\textfrak{F}^{L}(X, F)$ with $\lim\limits_{n} Lip\left(T - T_{n}\right)=0$. To show that $BTA\in\textfrak{G}^{L}(X_{0},F_{0})$, by using the nonlinear composition ideal property there is a sequence $\left((BT_{n}A)\right)_{n\in\mathbb{N}}$ in $\textfrak{F}^{L}(X_{0},F_{0})$ such that $\lim\limits_{n} Lip\left(BTA - BT_{n}A \right)$\\=0, the condition $\bf (\widetilde{NOI_2})$ satisfied. \\
\end{proof}

\subsection{Lipschitz Compact Operators}
Recall that the definition of Lipschitz compact operators of A. Jim{\'e}nez--Vargas, J. M. Sepulcre, and Mois{\'e}s Villegas--Vallecillos, \cite{JAJM14} is as follows. A Lipschitz operator $T\in Lip(X,F)$ is called Lipschitz compact if its Lipschitz image is relatively compact in $F$.  

The class of all Lipschitz compact operators is denoted by $\textfrak{R}^{L}$. By \cite{JAJM14} the following result is evident.

\begin{lem}
$\textfrak{R}^{L}$ is a nonlinear ideal. 
\end{lem}

\subsection{Lipschitz Weakly Compact Operators}
Recall that the definition of Lipschitz weakly compact operators of A. Jim{\'e}nez--Vargas, J. M. Sepulcre, and Mois{\'e}s Villegas--Vallecillos, \cite{JAJM14} is as follows. A Lipschitz operator $T\in Lip(X,F)$ is called Lipschitz weakly compact if its Lipschitz image is relatively weakly compact in $F$. 

The class of all Lipschitz compact operators is denoted by $\textfrak{W}^{L}$. By \cite{JAJM14} the following result is evident.

\begin{lem}
$\textfrak{W}^{L}$ is a nonlinear ideal. 
\end{lem}

\subsection{Order Relations between Nonlinear Ideals}
The order relations are collected in the following chain of inclusions.
$$smallest:=\textfrak{F}^{L}\subset\textfrak{G}^{L}\subset\textfrak{R}^{L}\subset\textfrak{W}^{L}\subset {Lip}=:largest.$$

\subsection{$\mathsf{A}$--Factorable Lipschitz Operators}\label{Auto17}
Let $\mathsf{A}$ be a space ideal. A Lipschitz operator $T\in Lip(X,F)$ is called $\mathsf{A}$--factorable if there exists a factorization $T=BA$ such that $A\in Lip(X,\mathbf{M})$, $B\in\mathfrak{L}(\mathbf{M},F)$, and $\mathbf{M}\in\mathsf{A}$. 

The class of all $\mathsf{A}$--factorable Lipschitz Operators is denoted by $Op^{L}(\mathsf{A})$. 
\begin{prop}
$Op^{L}(\mathsf{A})$ is a nonlinear ideal. 
\end{prop}

\begin{proof}
The condition $\bf (\widetilde{NOI_0})$ satisfied. Since $\mathbb{K}\in\mathsf{A}$, we have $g\boxtimes e =(1\otimes e)\circ (g\boxtimes 1)\in Op_{(X, F)}^{L}(\mathsf{A})$, where $1\otimes e\in\mathfrak{L}(\mathbb{K},F)$ and $g\boxtimes 1\in Lip(X,\mathbb{K})$. 

To prove the condition $\bf (\widetilde{NOI_1})$, let $T_{i}\in Lip(X,F)$ be $\mathsf{A}$--factorable Lipschitz operators. Then $T_{i}=B_{i} A_{i}$ with $A_{i}\in Lip(X,\mathbf{M}_{i})$, $B_{i}\in\mathfrak{L}(\mathbf{M}_{i},F)$, and $\mathbf{M}_{i}\in\mathsf{A}$.

$$T_{1} + T_{2}=\left(B_{1}\circ Q_{1} + B_{2}\circ Q_{2}\right) \left(J_{1}\circ A_{1} + J_{2}\circ A_{2}\right).$$
Hence $T_{1} + T_{2}$ factors through $\mathbf{M}_{1}\times \mathbf{M}_{2}\in\mathsf{A}$. This proves that $T_{1} + T_{2}\in Op_{(X, F)}^{L}(\mathsf{A})$. 

Let $A\in Lip(X_{0},X)$, $T\in Op^{L}(\mathsf{A})$, and $B\in\mathfrak{L}(F,F_{0})$. Then $T$ admits a factorization 
$$T: X\stackrel{\widetilde{A}}{\longrightarrow} \mathbf{M}\stackrel{\widetilde{B}}{\longrightarrow} F,$$
where $\widetilde{B}\in\mathfrak{L}(\mathbf{M},F)$ and $\widetilde{A}\in Lip(X,\mathbf{M})$.  To show that $BTA\in Op^{L}$. We obtain $B\circ\widetilde{B}\in\mathfrak{L}\left(\mathbf{M}, F_{0}\right)$ and $\widetilde{A}\circ A\in Lip\left(X_{0}, \mathbf{M}\right)$. Hence the Lipschitz operator $BTA$ admits a factorization 
$$BTA: X_{0}\stackrel{\widetilde{\widetilde{A\,}}}{\longrightarrow} \mathbf{M}\stackrel{\widetilde{\widetilde{B\,}}}{\longrightarrow} F_{0},$$ where $\widetilde{\widetilde{B\,}}=B\circ\widetilde{B}$ and $\widetilde{\widetilde{A\,}}=\widetilde{A}\circ A$, hence the condition $\bf (\widetilde{NOI_2})$ satisfied. \\
\end{proof}






\subsection{Products of Nonlinear Ideals}
Let $\mathfrak{A}$ be an ideal and $\textfrak{A}^{L}$ be nonlinear ideals. A Lipschitz operator $T\in Lip(X,F)$ belongs to the product $\mathfrak{A}\circ\textfrak{A}^{L}$ if there is a factorization $T=B\circ A$ with $B\in\mathfrak{A}\left(M, F\right)$ and $A\in\textfrak{A}^{L}\left(X, M\right)$. Here $M$ is a suitable Banach space.

\begin{prop}\label{Auto8}
$\mathfrak{A}\circ\textfrak{A}^{L}$ is a nonlinear ideal. 
\end{prop}

\begin{proof}
The condition $\bf (\widetilde{NOI_0})$ satisfied. Since the elementary Lipschitz tensor $g\boxtimes e$ admits a factorization $$g\boxtimes e : X\stackrel{g\boxtimes 1}{\longrightarrow} \mathbb{K}\stackrel{1\otimes e}{\longrightarrow} F,$$
where $1\otimes e\in\mathfrak{A}\left(\mathbb{K}, F\right)$ and $g\boxtimes 1\in\textfrak{A}^{L}\left(X, \mathbb{K}\right)$. 

To prove the condition $\bf (\widetilde{NOI_1})$, let $T_{i}\in\mathfrak{A}\circ\textfrak{A}^{L}(X, F)$. Then $T_{i}=B_{i}\circ A_{i}$ with $B_{i}\in\mathfrak{A}\left(M_{i}, F\right)$ and $A_{i}\in\textfrak{A}^{L}\left(X, M_{i}\right)$. Put $B:=B_{1}\circ Q_{1} + B_{2}\circ Q_{2}$, $A:=J_{1}\circ A_{1} + J_{2}\circ A_{2}$, and $M:=M_{1}\times M_{2}$. Now $T_{1} + T_{2}= B\circ A$, $B\in\mathfrak{A}\left(M, F\right)$ and $A\in\textfrak{A}^{L}\left(X, M\right)$ imply $T_{1} + T_{2}\in\mathfrak{A}\circ\textfrak{A}^{L}(X, F)$. 

Let $A\in Lip(X_{0},X)$, $T\in\mathfrak{A}\circ\textfrak{A}^{L}(X,F)$, and $B\in\mathfrak{L}(F,F_{0})$. Then $T$ admits a factorization 
$$T: X\stackrel{\widetilde{A}}{\longrightarrow} M\stackrel{\widetilde{B}}{\longrightarrow} F,$$
where $\widetilde{B}\in\mathfrak{A}\left(M, F\right)$ and $\widetilde{A}\in\textfrak{A}^{L}\left(X, M\right)$.  To show that $BTA\in\mathfrak{A}\circ\textfrak{A}^{L}(X_{0},F_{0})$. By using non--linear composition ideal properties, we obtain $B\circ\widetilde{B}\in\mathfrak{A}\left(M, F_{0}\right)$ and $\widetilde{A}\circ A\in\textfrak{A}^{L}\left(X_{0}, M\right)$. Hence the Lipschitz operator $BTA$ admits a factorization 
$$BTA: X_{0}\stackrel{\widetilde{\widetilde{A\,}}}{\longrightarrow} M\stackrel{\widetilde{\widetilde{B\,}}}{\longrightarrow} F_{0},$$ where $\widetilde{\widetilde{B\,}}=B\circ\widetilde{B}$ and $\widetilde{\widetilde{A\,}}=\widetilde{A}\circ A$, hence the condition $\bf (\widetilde{NOI_2})$ satisfied. \\
\end{proof}

We raise the following problem which we think is interesting.

\begin{OP}
Is it true that $\mathfrak{B}\circ\textfrak{W}^{L}=\textfrak{R}^{L} ?$
\end{OP}



\subsection{Quotients of Nonlinear Ideals}
Let $\mathfrak{A}$ be an operator ideal and $\textfrak{A}^{L}$ be a nonlinear ideal. A Lipschitz operator $T\in Lip(X,F)$ belongs to the quotient $\mathfrak{A}^{-1}\circ\textfrak{A}^{L}$ if
$B\circ T\in\textfrak{A}^{L}(X,F_{0})$ for all $B\in\mathfrak{A}(F,F_{0})$, where $F_{0}$ is an arbitrary Banach space. 


\begin{rem}
The single symbol $\mathfrak{A}^{-1}$ is without any meaning.
\end{rem}

\begin{prop}\label{Auto13}
$\mathfrak{A}^{-1}\circ\textfrak{A}^{L}$ is a nonlinear ideal. 
\end{prop}

\begin{proof}
Let $g\in X^{\#}$, $e\in F$ and $B$ be an arbitrary operator in $\mathfrak{A}(F,F_{0})$. Since $g\boxtimes e\in\textfrak{A}^{L}(X,F)$ and from nonlinear composition ideal property, we have $B\circ\left(g\boxtimes e\right)\in\textfrak{A}^{L}\left(X, F_{0}\right)$, hence $g\boxtimes e\in\textfrak{Y}^{L}$. 

Let $T_{i}\in\textfrak{Y}^{L}$ and let $B$ be an arbitrary operator in $\mathfrak{A}(F,F_{0})$. To prove the condition $\bf (\widetilde{NOI_1})$, i.e. to show that $B\circ \left(T_{1} + T_{2} \right)\in\textfrak{A}^{L}(X,F_{0})$. Let $x$ be an arbitrary element in $X$, from the linearity of the operator $B$ and Proposition \ref{Auto4}, we have  
\begin{align}
\left[B\circ \left(T_{1} + T_{2} \right)\right](x)&=B\left[\left(T_{1} + T_{2} \right)(x)\right]   \nonumber \\
&=B\left[T_{1}(x) + T_{2}(x)\right]                                                                \nonumber \\
&=B\left(T_{1} x\right) + B\left(T_{2} x\right)                                                    \nonumber \\
&=\left(B\circ T_{1}\right) (x) + \left(B\circ T_{2}\right) (x).                                   \nonumber 
\end{align}
Hence $B\circ \left(T_{1} + T_{2} \right)=B\circ T_{1} + B\circ T_{2}$. 

Let $A\in Lip(X_{0},X)$, $T\in\textfrak{Y}^{L}(X,F)$, and $B\in\mathfrak{L}(F,F_{0})$. Then $\widetilde{B}\circ T\in\textfrak{A}^{L}(X,G)$ for all $\widetilde{B}\in\mathfrak{A}(F,G)$, where $G$ is an arbitrary Banach space. To show that $BTA\in\textfrak{Y}^{L}(X_{0},F_{0})$. Let $\widetilde{\widetilde{B\,}}$ be an arbitrary operator in $\mathfrak{A}(F_{0},G)$, by using the non--linear composition ideal property, and the aforementioned assumption we have $\widetilde{\widetilde{B\,}}\circ B\in\mathfrak{A}(F,G)$, $(\widetilde{\widetilde{B\,}}\circ B)\circ T\in\textfrak{A}^{L}(X,G)$ and $\left[(\widetilde{\widetilde{B\,}}\circ B)\circ T\right]\circ A\in\textfrak{A}^{L}(X_{0},G)$. Hence $\widetilde{\widetilde{B\,}}\circ\left(BTA\right)\in\textfrak{A}^{L}(X_{0},G)$. The condition $\bf (\widetilde{NOI_2})$ satisfied. \\ 
\end{proof}

\begin{rem}
\begin{itemize}
	\item  A Lipschitz operator belonging to the nonlinear ideal $\mathfrak{B}^{-1}\circ\textfrak{R}^{L}$ should be called a \textbf{Rosenthal Lipschitz operator}.
	\item  A Lipschitz operator belonging to the nonlinear ideal $\mathfrak{X}^{-1}\circ\textfrak{W}^{L}$ should be called a \textbf{Grothendieck Lipschitz operator}.
  
\end{itemize}
\end{rem}
\section{Operator Ideals with Special Properties}

\subsection{Lipschitz Procedures}
A rule $$new: \mathfrak{A}\longrightarrow \mathfrak{A}^{L}_{new}$$

which defines a new nonlinear ideal $\mathfrak{A}^{L}_{new}$ for every ideal  $\mathfrak{A}$ is called a \textbf{semi--Lipschitz procedure}.

A rule $$new : \textswab{A}^{L}\longrightarrow \textswab{A}^{L}_{new}$$

which defines a new nonlinear ideal $\textswab{A}^{L}_{new}$ for every nonlinear ideal $\textswab{A}^{L}$ is called a \textbf{Lipschitz procedure}.

A rule $$new : \textfrak{A}^{L}\longrightarrow \textfrak{A}^{L}_{new}$$

which defines a new nonlinear ideal $\textfrak{A}^{L}_{new}$ for every nonlinear ideal  $\textfrak{A}^{L}$ is called a \textbf{Lipschitz procedure}.

\begin{rem}
We now list the following special property:

\begin{description}
	\item[$\bf (M)$] If $\textfrak{A}^{L}\subseteq\textfrak{B}^{L}$, then $\textfrak{A}^{L}_{new}\subseteq\textfrak{B}^{L}_{new}$  $\left(\textbf{strong monotony}\right)$.
	
	\item[$\bf (M')$] If $\mathfrak{A}\subseteq\mathfrak{B}$, then $\textfrak{A}^{L}_{new}\subseteq\textfrak{B}^{L}_{new}$  $\left(\textbf{monotony}\right)$.
	
	\item[$\bf (I)$] $\left(\textfrak{A}^{L}_{new}\right)_{new}=\textfrak{A}^{L}_{new}$ for all $\textfrak{A}^{L}$  $\left(\textbf{idempotence}\right)$.
	
\end{description}

A strong monotone and idempotent Lipschitz procedure is called a \textbf{hull Lipschitz procedure} if $\textfrak{A}^{L}\subseteq\textfrak{A}^{L}_{new}$. 
\end{rem}
\subsection{Closed Nonlinear Ideals}
Let $\textfrak{A}^{L}$ be a nonlinear ideal. A Lipschitz operator $T\in Lip(X,F)$ belongs to the closure $\textfrak{A}^{L}_{clos}$ if there are $T_{1}, T_{2}, T_{3},\cdots\in\textfrak{A}^{L}(X, F)$ with $\lim\limits_{n} Lip\left(T - T_{n}\right)=0$. 

\begin{prop}
$\textfrak{A}^{L}_{clos}$ is a nonlinear ideal. 
\end{prop}
 
\begin{proof}
The condition $\bf (\widetilde{NOI_0})$ satisfied. Since there is a sequence $\left(T_{n}\right)_{n\in\mathbb{N}}=g\boxtimes e$ in $\textfrak{A}^{L}(X, F)$ with $\lim\limits_{n} Lip\left(g\boxtimes e - T_{n}\right)=0$. To prove the condition $\bf (\widetilde{NOI_1})$, let $T$ and $S$ in $\textfrak{A}^{L}_{clos}(X,F)$. Then there are sequences $\left(T_{n}\right)_{n\in\mathbb{N}}$ and $\left(S_{n}\right)_{n\in\mathbb{N}}$ in $\textfrak{A}^{L}(X, F)$ with $\lim\limits_{n} Lip\left(T - T_{n}\right)=\lim\limits_{n} Lip\left(S - S_{n}\right)=0$. From Proposition \ref{Auto4} there is a sequence $\left((T_{n} + S_{n})\right)_{n\in\mathbb{N}}$ in $\textfrak{A}^{L}(X, F)$ such that $\lim\limits_{n} Lip\left(T + S - (T_{n} + S_{n})\right)=0$.

Let $A\in Lip(X_{0},X)$, $T\in\textfrak{A}^{L}_{clos}(X, F)$, and $B\in\mathfrak{L}(F,F_{0})$. Then there is a sequence $\left(T_{n}\right)_{n\in\mathbb{N}}$ in $\textfrak{A}^{L}(X, F)$ with $\lim\limits_{n} Lip\left(T - T_{n}\right)=0$. To show that $BTA\in\textfrak{A}^{L}_{clos}(X_{0},F_{0})$, by using the nonlinear composition ideal property there is a sequence $\left((BT_{n}A)\right)_{n\in\mathbb{N}}$ in $\textfrak{A}^{L}(X_{0},F_{0})$ such that \\ $\lim\limits_{n} Lip\left(BTA - BT_{n}A \right)=0$, the condition $\bf (\widetilde{NOI_2})$ satisfied. \\
\end{proof}

The following statement is evident.

\begin{prop}
The rule $$clos: \textfrak{A}^{L}\longrightarrow\textfrak{A}^{L}_{clos}$$
is a hull Lipschitz procedure. 
\end{prop}

\begin{rem}
The nonlinear ideal $\textfrak{A}^{L}$ is called closed if $\textfrak{A}^{L}=\textfrak{A}^{L}_{clos}$.
\end{rem} 

\begin{lem}
$\textfrak{G}^{L}$ is the smallest closed nonlinear ideals. 
\end{lem}

\begin{proof}
By the definition of Lipschitz approximable operator in (\ref{Auto 15}) we have $\textfrak{G}^{L}:=\textfrak{F}^{L}_{clos}$.

\end{proof} 

\begin{lem}\label{Auto18}
Let $\textfrak{A}^{L}$ be a closed nonlinear ideal. Then $\textfrak{G}^{L}\subseteq\textfrak{A}^{L}$. 									 
\end{lem}

\begin{proof}
Suppose $T\in\textfrak{G}^{L}$, hence $T\in\textfrak{A}^{L}_{clos}$. Since $\textfrak{A}^{L}$ is closed we have $T\in\textfrak{A}^{L}$.  \\
\end{proof} 

\subsection{Radical}
Let $\textswab{A}^{L}$ be a nonlinear operator. A Lipschitz operator $T\in Lip(E,F)$ belongs to the radical $\textswab{A}^{L}_{rad}$ if for every $L\in\mathfrak{L}(F,E)$ there exist $U\in\mathfrak{L}(E)$ and $S\in\textswab{A}^{L}(E)$ such that 
\begin{equation}\label{Auto10}
U\left(I_{E} - L T\right)= I_{E} - S.
\end{equation}
\begin{prop}
$\textswab{A}^{L}_{rad}$ is a nonlinear ideal. 
\end{prop}

\begin{proof}
The condition $\bf (NOI_0)$ satisfied. Since for arbitrary operator $L$ in $\mathfrak{L}(F,E)$, put $U:= I_{E}$ and $S:= L\circ\left(h\boxtimes e\right)$. Then $S\in\textswab{A}^{L}(E)$ such that (\ref{Auto10}) fulfilled. To prove the condition $\bf (NOI_1)$, let $T_{1}$ and $T_{2}$ in $\textswab{A}^{L}_{rad}(E,F)$. Then, given $L\in\mathfrak{L}(F,E)$, there are $U_{1}\in\mathfrak{L}(E)$ and $S_{1}\in\textswab{A}^{L}(E)$ with 
$U_{1}\left(I_{E} - L T_{1}\right) = I_{E} - S_{1}$. We now choose $U_{2}\in\mathfrak{L}(E)$ and $S_{2}\in\textswab{A}^{L}(E)$ such that $U_{2}\left(I_{E} - U_{1} L T_{2}\right) = I_{E} - S_{2}$. Then $$U_{2} U_{1}\left[I_{E} - L (T_{1} + T_{1})\right]=U_{2}\left[I_{E} - S_{1}- U_{1} L T_{2}\right]=I_{E} - S_{2} - U_{2} S_{1}.$$ 
Since $S_{2} + U_{2} S_{1}\in\textswab{A}^{L}(E)$, we have $T_{1} + T_{2}\in\textswab{A}^{L}_{rad}(E,F)$. 

Let $A\in Lip(E_{0},E)$, $T\in\textswab{A}^{L}_{rad}(E,F)$, and $B\in\mathfrak{L}(F,F_{0})$. Given $L\in\mathfrak{L}(F_{0},E_{0})$, there exist $U\in\mathfrak{L}(E)$ and $S\in\textswab{A}^{L}(E)$ with $U\left(I_{E} - A L B T\right)= I_{E} - S$. Define  the operators $U_{0}:=I_{E_{0}}+LBTUA$ and $S_{0}:=LBTSA$. Clearly $S_{0}\in\textswab{A}^{L}(E_{0})$ and 
\begin{align}
U_{0}\left(I_{E_{0}} - LBTA\right)&= I_{E_{0}} - LBTA + LBTU\left(I_{E} - ALBT\right)A   \nonumber \\
&=I_{E_{0}} - LBTA + LBT\left(I_{E} - S\right)A                                         \nonumber \\
&=I_{E_{0}} - LBTSA = I_{E} - S_{0}.                                                    \nonumber 
\end{align}

Therefore  $BTA\in\textswab{A}^{L}_{rad}(E_{0},F_{0})$. \\
\end{proof}

\begin{prop}
The rule $$rad: \textswab{A}^{L}\longrightarrow\textswab{A}^{L}_{rad}$$
is a hull Lipschitz procedure. 
\end{prop}

\begin{proof}
The property $\bf (M)$ is obvious. If $T\in\textswab{A}^{L}(E, F)$ and $L\in\mathfrak{L}(F,E)$, put $U:= I_{E}$ and $S:= L\circ T$. Then 
$S\in\textswab{A}^{L}(E)$ such that (\ref{Auto10}) hold. Consequently $T\in\textswab{A}^{L}_{rad}(E, F)$. This proves that $\textswab{A}^{L}\subseteq\textswab{A}^{L}_{rad}$. To prove property $\bf (I)$, let $T\in Lip(E, F)$ belong to $\left(\textswab{A}^{L}_{rad}\right)_{rad}$. Then, given $L\in\mathfrak{L}(F,E)$, there are $U_{1}\in\mathfrak{L}(E,E)$ and $S_{1}\in\textswab{A}^{L}_{rad}(E)$ such that $U_{1}\left(I_{E} - L T\right)= I_{E} - S_{1}$. We can now find $U_{2}\in\mathfrak{L}(E,E)$ and $S_{2}\in\textswab{A}^{L}(E)$ with $U_{2}\left(I_{E} - S_{1}\right)= I_{E} - S_{2}$. Consequently $$U_{2} U_{1}\left(I_{E} - L T\right)= U_{2}\left(I_{E} - S_{1}\right)= I_{E} - S_{2}.$$
This implies that $T\in\textswab{A}^{L}_{rad}$. Therefore $\left(\textswab{A}^{L}_{rad}\right)_{rad}\subseteq\textswab{A}^{L}_{rad}$. The converse inclusion is trivial. \\
\end{proof}

Finally, we show that the symmetry in the definition of $\textswab{A}^{L}$ can be removed.

\begin{lem}\label{Auto24}
Let $T\in\textswab{A}^{L}_{rad}(E, F)$. Then for every $L\in\mathfrak{L}(F,E)$ there exist $U\in\mathfrak{L}(E)$ and $S_{1}, S_{2}\in\textswab{A}^{L}(E)$ such that 
$$ U\left(I_{E} - L T\right)= I_{E} - S_{1} \ \ \ and \ \ \  \left(I_{E} - L T\right) U= I_{E} - S_{2}.$$
\end{lem}

\begin{proof}
We can find $U\in\mathfrak{L}(E)$ and $S_{1}\in\textswab{A}^{L}(E)$ with $U\left(I_{E} - L T\right)= I_{E} - S_{1}$.\\ Since 
$$R:=I_{E} - U=S_{1}-ULT\in\textswab{A}^{L}_{rad}(E),$$ there are $U_{0}\in\mathfrak{L}(E)$ and $S_{0}\in\textswab{A}^{L}(E)$ such that $U_{0}\left(I_{E} - R\right)= I_{E} - S_{0}$. We obtain from $U_{0} U + S_{0}= I_{E}$ that 
\begin{align}
\left(I_{E} - LT\right) U&=U_{0} U\left( I_{E} - LT\right) U + S_{0}\left(I_{E} - L T\right)U   \nonumber \\
&=U_{0} \left( I_{E} - S_{1}\right) U + S_{0}\left(I_{E} - L T\right)U= I_{E}-S_{2},             \nonumber 
\end{align}
where $$S_{2}:=S_{0}\left(I_{E} -U+ LTU\right)+ U_{0}S_{1}U\in\textswab{A}^{L}(E).$$
\end{proof}

\begin{lem}
Let $T\in\textswab{A}^{L}_{rad}(E, F)$. Then for every $L\in\mathfrak{L}(F,E)$ there exist $V\in\mathfrak{L}(F)$ and $P_{1}, P_{2}\in\textswab{A}^{L}(F)$ such that 
$$ V\left(I_{F} - T L\right)= I_{F} - P_{1} \ \ \ and \ \ \  \left(I_{F} - T L\right) V= I_{F} - P_{2}.$$
\end{lem}

\begin{proof}
Apply Lemma \ref{Auto24} to $T L\in\textswab{A}^{L}(F)$.\\
\end{proof}
\subsection{Symmetric Nonlinear Ideals}\label{Auto19}
Let $\mathfrak{A}$ be an ideal. A Lipschitz operator $T\in Lip(X,F)$ belongs to the Lipschitz dual ideal $\mathfrak{A}^{L}_{dual}$ if $T^{\#}_{|_{{F}^{*}}}\in\mathfrak{A}(F^{*}, X^{\#})$.

\begin{lem}\label{Auto24}
Let $T$ be in $\textfrak{F}^{L}(X, F)$, with $T=\sum\limits_{j=1}^{m} g_{j}\boxtimes e_{j}$. Then $T^{\#}_{|_{{F}^{*}}}=\sum\limits_{j=1}^{m} \hat{e}_{j}\otimes g_{j}$, where $e\longmapsto \hat{e}$ is the natural embedding of a space $F$ into its second dual $F^{**}$.
\end{lem}

\begin{proof}
We have $Tx=\sum\limits_{j=1}^{m} g_{j}(x) e_{j}$ for $x\in X$. So for $y^{*}\in F^{*}$,
$$\left\langle T^{\#}_{|_{{F}^{*}}} y^{*},x\right\rangle_{(X^{\#},X)}=\left\langle y^{*}, Tx\right\rangle_{(F^{*},F)}=\sum\limits_{j=1}^{m}g_{j}(x)y^{*} (e_{j}).$$ Hence $T^{\#}_{|_{{F}^{*}}} y^{*}=\sum\limits_{j=1}^{m} y^{*}(e_{j}) g_{j}$. This proves the statement for $T^{\#}_{|_{{F}^{*}}}$.
\end{proof}

\begin{lem}\label{Auto16}
Let $T, S\in Lip(X,F)$, $A\in Lip(X_{0},X)$, and $B\in\mathfrak{L}(F,F_{0})$. Then

\begin{itemize}
	\item $\left(T+S\right)^{\#}_{|_{{F}^{*}}}=T^{\#}_{|_{{F}^{*}}} + S^{\#}_{|_{{F}^{*}}}$.
	\item $\left(BTA\right)^{\#}_{|_{F^{*}_{0}}}=A^{\#}T^{\#}_{|_{{F}^{*}}}B^{*}$.
\end{itemize}
\end{lem}

\begin{proof}
For  $y^{*}\in F^{*}$ and $x\in X$, we have
\begin{align}
\left\langle \left(T+S\right)^{\#}_{|_{{F}^{*}}} y^{*},x\right\rangle_{(X^{\#},X)}&=\left\langle y^{*}, (T+S)x\right\rangle_{(F^{*},F)}  \nonumber \\
&=\left\langle y^{*}, Tx+Sx\right\rangle_{(F^{*},F)} \nonumber \\
&=\left\langle y^{*}, Tx\right\rangle_{(F^{*},F)} + \left\langle y^{*}, Sx\right\rangle_{(F^{*},F)}\nonumber \\
&=\left\langle T^{\#}_{|_{{F}^{*}}} y^{*},x\right\rangle_{(X^{\#},X)} + \left\langle S^{\#}_{|_{{F}^{*}}} y^{*},x\right\rangle_{(X^{\#},X)} \nonumber
\end{align}
Hence $\left(T+S\right)^{\#}_{|_{{F}^{*}}}=T^{\#}_{|_{{F}^{*}}} + S^{\#}_{|_{{F}^{*}}}$. 

For  $y^{*}_{0}\in F^{*}$ and $x_{0}\in X_{0}$, we have
\begin{align}
\left\langle y_{0}^{*}, BTA (x_{0})\right\rangle_{(F_{0}^{*},F_{0})}&=\left\langle y_{0}^{*}, B(TA x_{0})\right\rangle_{(F_{0}^{*},F_{0})}  \nonumber \\
&=\left\langle B^{*} y_{0}^{*}, T(A x_{0})\right\rangle_{(F^{*},F)} \nonumber \\
&=\left\langle T^{\#}_{|_{{F}^{*}}} B^{*} y_{0}^{*}, A x_{0}\right\rangle_{(X^{\#}, X)} \nonumber \\
&=\left\langle A^{\#} T^{\#}_{|_{{F}^{*}}} B^{*} y_{0}^{*}, x_{0}\right\rangle_{(X_{0}^{\#}, X_{0})} \nonumber 
\end{align}
But also $\left\langle y_{0}^{*}, BTA (x)\right\rangle_{(F_{0}^{*},F_{0})}=\left\langle \left(BTA\right)^{\#}_{|_{F^{*}_{0}}} y_{0}^{*}, x_{0}\right\rangle_{(X_{0}^{\#}, X_{0})}$. Therefore $\left(BTA\right)^{\#}_{|_{F^{*}_{0}}}=A^{\#}T^{\#}_{|_{{F}^{*}}}B^{*}$. \\
\end{proof}

\begin{prop}
$\mathfrak{A}^{L}_{dual}$ is a nonlinear ideal. 
\end{prop}

\begin{proof}
The condition $\bf (\widetilde{NOI_0})$ satisfied. From Lemma \ref{Auto24} we obtain $\left(g\boxtimes e\right)^{\#}_{|_{{F}^{*}}}=\hat{e}\otimes g\in\mathfrak{A}(F^{*}, X^{\#})$. To prove the condition $\bf (\widetilde{NOI_1})$, let $T$ and $S$ in $\mathfrak{A}^{L}_{dual}(X,F)$. Let $T^{\#}_{|_{{F}^{*}}}$ and $S^{\#}_{|_{{F}^{*}}}$ in $\mathfrak{A}(F^{*}, X^{\#})$, from Lemma \ref {Auto16} we have $\left(T+S\right)^{\#}_{|_{{F}^{*}}}=T^{\#}_{|_{{F}^{*}}} + S^{\#}_{|_{{F}^{*}}}\in\mathfrak{A}(F^{*}, X^{\#})$. 

Let $A\in Lip(X_{0},X)$, $T\in\mathfrak{A}^{L}_{dual}(X,F)$, and $B\in\mathfrak{L}(F,F_{0})$. Also from Lemma \ref {Auto16} we have $\left(BTA\right)^{\#}_{|_{F^{*}_{0}}}=A^{\#}T^{\#}_{|_{{F}^{*}}}B^{*}\in\mathfrak{A}(F_{0}^{*}, X_{0}^{\#})$, hence the condition $\bf (\widetilde{NOI_2})$ satisfied. \\
\end{proof}

\begin{prop}
The rule $$dual: \mathfrak{A}\longrightarrow \mathfrak{A}^{L}_{dual}$$
is a monotone Lipschitz procedure. 
\end{prop}

\begin{rem}
A nonlinear ideal $\textfrak{A}^{L}$ is called \textbf{symmetric} if $\textfrak{A}^{L}\subseteq\mathfrak{A}^{L}_{dual}$. In case $\textfrak{A}^{L}=\mathfrak{A}^{L}_{dual}$ the nonlinear ideal is said to be \textbf{completely symmetric}.
\end{rem} 

\begin{lem}
The nonlinear ideal $\textfrak{F}^{L}$ is completely symmetric.
\end{lem}

\begin{proof}
Let $T\in\textfrak{F}^{L}(X, F)$, then $T$ can be representation in the form $\sum\limits_{j=1}^{m} g_{j}\boxtimes e_{j}$. From Lemma \ref{Auto24} we have $T^{\#}_{|_{{F}^{*}}}=\sum\limits_{j=1}^{m}\hat{e}_{j}\otimes g_{j}\in F^{**}\otimes X^{\#}\equiv\mathfrak{F}(F^{*}, X^{\#})$. Hence $T\in\mathfrak{F}^{L}_{dual}(X, F)$. 

Let $T\in\mathfrak{F}^{L}_{dual}(X, F)$ then $T^{\#}_{|_{{F}^{*}}}\in\mathfrak{F}(F^{*}, X^{\#})$ hence $T^{\#}_{|_{{F}^{*}}}$ can be representation in the form $\sum\limits_{j=1}^{m} \hat{e}_{j}\otimes g_{j}$. For  $y^{*}\in F^{*}$ and $x\in X$, we have
\begin{align}
\left\langle y^{*}, T x\right\rangle_{(F^{*},F)}&=\left\langle T^{\#}_{|_{{F}^{*}}} y^{*}, x\right\rangle_{(X^{\#}, X)}  \nonumber \\
&=\left\langle \sum\limits_{j=1}^{m} \hat{e}_{j}\otimes g_{j}\  (y^{*}), x\right\rangle_{(X^{\#}, X)}  \nonumber \\
&=\left\langle \sum\limits_{j=1}^{m} \hat{e}_{j}(y^{*})\cdot g_{j}\  , x\right\rangle_{(X^{\#}, X)}  \nonumber \\
&=\left\langle \sum\limits_{j=1}^{m} y^{*}({e}_{j})\cdot g_{j}\  , x\right\rangle_{(X^{\#}, X)}  \nonumber \\
&=\sum\limits_{j=1}^{m} g_{j}(x)\cdot y^{*}(e_{j})  \nonumber \\
&=\left\langle y^{*}, \sum\limits_{j=1}^{m} g_{j}\boxtimes e_{j} \ (x)\right\rangle_{(F^{*},F)}. \nonumber
\end{align}
Hence $T=\sum\limits_{j=1}^{m} g_{j}\boxtimes e_{j}\in\textfrak{F}^{L}(X, F)$. \\
\end{proof}

\subsection{Regular Nonlinear Ideals}
Let $\textfrak{A}^{L}$ be a nonlinear ideal. A Lipschitz operator $T\in Lip(X,F)$ belongs to the regular hull $\textfrak{A}^{L}_{reg}$ if $K_{F}T\in\textfrak{A}^{L}(X, F^{**})$.

\begin{prop}
$\textfrak{A}^{L}_{reg}$ is a nonlinear ideal. 
\end{prop}

\begin{proof}
The condition $\bf (\widetilde{NOI_0})$ satisfied. Since $g\boxtimes e\in\textfrak{A}^{L}(X, F)$ and using nonlinear composition ideal property we have $K_{F}(g\boxtimes e)\in\textfrak{A}^{L}(X, F^{**})$. To prove the condition $\bf (\widetilde{NOI_1})$, let $T$ and $S$ in $\textfrak{A}^{L}_{reg}(X,F)$. Then $K_{F}T$ and $K_{F}S$ in $\textfrak{A}^{L}(X, F^{**})$, we have $K_{F}(T+S)=K_{F}T + K_{F}S\in\textfrak{A}^{L}(X, F^{**})$. 

Let $A\in Lip(X_{0},X)$, $T\in\textfrak{A}^{L}_{reg}(X, F)$, and $B\in\mathfrak{L}(F,F_{0})$. Then
$$
\begin{tikzcd}[row sep=5.0em, column sep=5.0em]
X   \arrow{r}{T}                & F   \arrow{r}{K_F} \arrow{d}{B} & F^{**} \arrow{d}{B^{**}} \\
X_0 \arrow{r}{BTA} \arrow{u}{A} & F_0 \arrow{r}{K_{F_0}}          & F_0^{**}                  \\
\end{tikzcd}
\vspace{-50pt}
$$
Consequently $K_{F_{0}}\left(BTA\right)=B^{**}\left(K_{F}T\right)A\in\textfrak{A}^{L}$, hence the condition $\bf (\widetilde{NOI_2})$ satisfied. \\
\end{proof}

\begin{prop}
The rule $$reg: \textfrak{A}^{L}\longrightarrow\textfrak{A}^{L}_{reg}$$
is a hull Lipschitz procedure. 
\end{prop}

\begin{proof}
The property $\bf (M)$ is obvious. From the nonlinear composition ideal we obtain $\textfrak{A}^{L}\subseteq\textfrak{A}^{L}_{reg}$. To show the idempotence. Let $T\in Lip(X,F)$ belong to $\left(\textfrak{A}^{L}_{reg}\right)_{reg}$. Then $K_{F}T\in\textfrak{A}^{L}_{reg}(X, F^{**})$ and $K_{F^{**}} K_{F} T\in\textfrak{A}^{L}(X, F^{****})$. Now $I_{F^{**}}=(K_{F^{*}})^{*} K_{F^{**}}$ implies $$K_{F}T=(K_{F^{*}})^{*}\left(K_{F^{**}} K_{F} T\right)\in\textfrak{A}^{L}(X, F^{**})$$ and therefore $T\in\textfrak{A}^{L}_{reg}(X, F)$. Hence $\left(\textfrak{A}^{L}_{reg}\right)_{reg}\subseteq\textfrak{A}^{L}_{reg}$. The converse inclusion is trivial. \\
\end{proof}

\begin{rem}
A nonlinear ideal $\textfrak{A}^{L}$ is called regular if $\textfrak{A}^{L}=\textfrak{A}^{L}_{reg}$.
\end{rem} 

\begin{prop}\label{Auto25}
Let $\mathfrak{A}$ be an ideal. Then $\mathfrak{A}^{L}_{dual}$ is regular.
\end{prop}

\begin{proof}
We check only the inclusion $\left(\mathfrak{A}^{L}_{dual}\right)_{reg}\subseteq\mathfrak{A}^{L}_{dual}$. Let $T\in Lip(X,F)$ belong to $\left(\mathfrak{A}^{L}_{dual}\right)_{reg}$. Then $K_{F}T\in\mathfrak{A}^{L}_{dual}(X, F^{**})$ and $T^{\#}_{|_{{F}^{*}}} (K_{F})^{*}\in\mathfrak{A}(F^{***}, X^{\#})$. Now $I_{F^{*}}=(K_{F})^{*} K_{F^{*}}$ implies $T^{\#}_{|_{{F}^{*}}}=T^{\#}_{|_{{F}^{*}}} (K_{F})^{*} K_{F^{*}}\in\mathfrak{A}(F^{*}, X^{\#})$ then $T\in\mathfrak{A}^{L}_{dual}(X, F)$. \\
\end{proof}

The Proposition \ref{Auto25} gives the following result.
\begin{lem}
If $\mathfrak{A}$ be an ideal, then every completely symmetric nonlinear ideal is regular and symmetric.
\end{lem}

\begin{lem}
The nonlinear ideal $\textfrak{F}^{L}$ is regular.
\end{lem}

\begin{proof}
The regularity of $\textfrak{F}^{L}$ is implied by its completely symmetry. \\
\end{proof}

\subsection{Injective Nonlinear Ideals}\label{Auto20}
Let $\textfrak{A}^{L}$ be a nonlinear ideal. A Lipschitz operator $T\in Lip(X,F)$ belongs to the injective hull $\textfrak{A}^{L}_{inj}$ if $J_{F}T\in\textfrak{A}^{L}(X, F^{inj})$.

\begin{prop}
$\textfrak{A}^{L}_{inj}$ is a nonlinear ideal. 
\end{prop}

\begin{proof}
The condition $\bf (\widetilde{NOI_0})$ satisfied. Since $g\boxtimes e\in\textfrak{A}^{L}(X, F)$ and using nonlinear composition ideal property we have $J_{F}(g\boxtimes e)\in\textfrak{A}^{L}(X, F^{inj})$. To prove the condition $\bf (\widetilde{NOI_1})$, let $T$ and $S$ in $\textfrak{A}^{L}_{inj}(X,F)$. Then $J_{F}T$ and $J_{F}S$ in $\textfrak{A}^{L}(X, F^{inj})$, we have $J_{F}(T+S)=J_{F}T + J_{F}S\in\textfrak{A}^{L}(X, F^{inj})$. 

Let $A\in Lip(X_{0},X)$, $T\in\textfrak{A}^{L}_{inj}(X, F)$, and $B\in\mathfrak{L}(F,F_{0})$. Since $F_{0}^{inj}$ has the extension property, there exists $B^{inj}\in\mathfrak{L}(F^{inj},F_{0}^{inj})$ such that

$$
\begin{tikzcd}[row sep=5.0em, column sep=5.0em]
X   \arrow{r}{T}                & F   \arrow{r}{J_F} \arrow{d}{B} & F^{inj} \arrow{d}{B^{inj}} \\
X_0 \arrow{r}{BTA} \arrow{u}{A} & F_0 \arrow{r}{J_{F_0}}          & F_0^{inj}                  \\
\end{tikzcd}
\vspace{-50pt}
$$

Consequently $J_{F_{0}}\left(BTA\right)=B^{inj}\left(J_{F}T\right)A\in\textfrak{A}^{L}$, hence the condition $\bf (\widetilde{NOI_2})$ satisfied. \\
\end{proof}

\begin{lem}\label{Auto21}
Let $F$ be a Banach space possessing the extension property. Then $\textfrak{A}^{L}(X, F)=\textfrak{A}^{L}_{inj}(X, F)$.
\end{lem}

\begin{proof}
By hypothesis there exists $B\in\mathfrak{L}(F^{inj},F)$ such that $BJ_{F}=I_{F}$. Therefore $T\in\textfrak{A}^{L}_{inj}(X, F)$ implies that $T=B\left(J_{F}T\right)\in\textfrak{A}^{L}(X, F)$. This proves that $\textfrak{A}^{L}_{inj}\subseteq\textfrak{A}^{L}$. The converse inclusion is obvious. \\
\end{proof}

\begin{prop}\label{Auto 14}
The rule $$inj: \textfrak{A}^{L}\longrightarrow\textfrak{A}^{L}_{inj}$$
is a hull Lipschitz procedure. 
\end{prop}

\begin{proof}
The property $\bf (M)$ is obvious. From the preceding lemma we obtain $\textfrak{A}^{L}\subseteq\textfrak{A}^{L}_{inj}$. To show the idempotence. Let $T\in Lip(X,F)$ belong to $\left(\textfrak{A}^{L}_{inj}\right)_{inj}$. Then $J_{F}T\in\textfrak{A}^{L}_{inj}(X, F^{inj})$, and the preceding lemma implies $J_{F}T\in\textfrak{A}^{L}(X, F^{inj})$. Consequently $T\in\textfrak{A}^{L}_{inj}(X, F)$. Thus $\left(\textfrak{A}^{L}_{inj}\right)_{inj}\subseteq\textfrak{A}^{L}_{inj}$. The converse inclusion is trivial. \\
\end{proof}

\begin{prop}
Let $\mathsf{A}$ be a space ideal. Then $\left[Op^{L}(\mathsf{A})\right]_{inj}\subseteq Op^{L}(\mathsf{A^{inj}})$.
\end{prop}

\begin{proof}
Let $T\in Lip(X,F)$ belong to $\left[Op^{L}(\mathsf{A})\right]_{inj}$. Then $J_{F}T\in Op^{L}(\mathsf{A})(X, F^{inj})$, so $J_{F} T=BA$ such that $A\in Lip(X,\mathbf{M})$, $B\in\mathfrak{L}(\mathbf{M},F^{inj})$, and $\mathbf{M}\in\mathsf{A}$. Put $\mathbf{M_{0}}:=\overline{R_{A}}$, and let $A_{0}\in Lip(X,\mathbf{M_{0}})$ be the Lipschitz operator induced by $A$. Obviously $B\mathbf{m}\in R_{J_{F}}$ for $\mathbf{m}\in R_{A}$. Consequently $B(\mathbf{M_{0}})\subseteq R_{J_{F}}$, and $B_{0}:=J_{F}^{-1} B J_{\mathbf{M_{0}}}^{\mathbf{M}}$  is well--defined. Finally, $T=B_{0} A_{0}$ and $\mathbf{M_{0}}\in\mathsf{A^{inj}}$ imply $T\in Op^{L}(\mathsf{A^{inj}})$. This proves $\left[Op^{L}(\mathsf{A})\right]_{inj}\subseteq  Op^{L}(\mathsf{A^{inj}})$.\\
\end{proof}

\begin{rem}
A nonlinear ideal $\textfrak{A}^{L}$ is called injective if $\textfrak{A}^{L}=\textfrak{A}^{L}_{inj}$.
\end{rem}

The injectivity of a nonlinear ideal $\textfrak{A}^{L}$ means that it does not depend on the size of the target space $F$ whether or not a Lipschitz operator $T\in Lip(X,F)$  belongs to $\textfrak{A}^{L}$.

\begin{prop}\label{Auto22}
A nonlinear ideal $\textfrak{A}^{L}$ is injective if and only if for every injection $J\in\mathfrak{L}(F_{0}, F)$ and for every Lipschitz operator $T_{0}\in Lip(X, F_{0})$ it follows from $JT_{0}\in\textfrak{A}^{L}(X, F)$ that $T_{0}\in\textfrak{A}^{L}(X, F_{0})$.
\end{prop}

\begin{proof}
To check the necessity we consider an injective nonlinear ideal $\textfrak{A}^{L}$. Let $T_{0}\in Lip(X, F_{0})$ such that $JT_{0}\in\textfrak{A}^{L}(X, F)$. Since $F_{0}^{inj}$ has the extension property, there is $B\in\mathfrak{L}(F, F_{0}^{inj})$ with $J_{F_{0}}=BJ$. Hence it follows from $J_{F_{0}} T_{0}=B\left(JT_{0}\right)\in\textfrak{A}^{L}(X, F_{0}^{inj})$ that $T_{0}\in\textfrak{A}^{L}_{inj}(X, F_{0})=\textfrak{A}^{L}(X, F_{0})$. 

Conversely, let us suppose that the given condition is satisfied. If $T_{0}\in\textfrak{A}^{L}_{inj}(X, F_{0})$, then $J_{F_{0}} T_{0}\in\textfrak{A}^{L}(X, F_{0}^{inj})$. Since $J_{F_{0}}$ is an injection, we obtain $T_{0}\in\textfrak{A}^{L}(X, F_{0})$. Therefore $\textfrak{A}^{L}=\textfrak{A}^{L}_{inj}$, which proves the sufficiency. \\
\end{proof}
\subsection{Surjective Nonlinear Ideals}
Let $\textswab{A}^{L}$ be a nonlinear operator. A Lipschitz operator $T\in Lip(E,F)$ belongs to the surjective hull $\textswab{A}^{L}_{sur}$ if $TQ_{E}\in\textswab{A}^{L}(E^{sur}, F)$.
\begin{prop}
$\textswab{A}^{L}_{sur}$ is a nonlinear ideal. 
\end{prop}

\begin{proof}
The condition $\bf (NOI_0)$ satisfied. Since $h\boxtimes e\in\textswab{A}^{L}(E, F)$ and using nonlinear composition ideal property we have $\left(h\boxtimes e\right) Q_{E}\in\textswab{A}^{L}(E^{sur}, F)$. To prove the condition $\bf (NOI_1)$, let $T_{1}$ and $T_{2}$ in $\textswab{A}^{L}_{sur}(E,F)$. Since $\left(T_{1}+T_{2}\right)Q_{E}=T_{1}Q_{E}+T_{2}Q_{E}$, we have $T_{1} + T_{2}\in\textswab{A}^{L}_{sur}(E,F)$. 

Let $A\in Lip(E_{0},E)$, $T\in\textswab{A}^{L}_{sur}(E,F)$, and $B\in\mathfrak{L}(F,F_{0})$. Since $E_{0}^{sur}$ has the lifting property, there exists $B^{sur}\in\mathfrak{L}(E_{0}^{sur}, E^{sur})$ such that 
$$
\begin{tikzcd}[row sep=5.0em, column sep=5.0em]
E^{sur}   \arrow{r}{Q_{E}}                & E  \arrow{r}{T} \arrow{d}{A} & F \arrow{d}{B} \\
E_{0}^{sur} \arrow{r}{Q_{E_{0}}} \arrow{u}{B^{sur}} & E_0 \arrow{r}{BTA}          & F_{0}                  \\
\end{tikzcd}
\vspace{-50pt}
$$

Consequently $\left(BTA\right) Q_{E_{0}}=B\left(T Q_{E}\right)B^{sur}\in\textfrak{A}^{L}(E_{0}^{sur}, F_{0})$, hence the condition $\bf (NOI_2)$.  \\
\end{proof}

\begin{lem}
Let $E$ be a Banach space with the lifting property. Then $\textswab{A}^{L}(E, F)=\textswab{A}^{L}_{sur}(E, F)$.
\end{lem}

\begin{proof}
By hypothesis there exists $B\in\mathfrak{L}(E, E^{sur})$ such that $Q_{E} B=I_{E}$. Therefore $T\in\textswab{A}^{L}_{sur}(E, F)$ implies that $T=\left(T Q_{E}\right)B\in\textswab{A}^{L}(E, F)$. This proves that $\textswab{A}^{L}_{sur}\subseteq\textswab{A}^{L}$. The converse inclusion is obvious. \\
\end{proof}

Similarly to Proposition \ref{Auto 14} and from the preceding lemma we have
\begin{prop}
The rule $$sur: \textswab{A}^{L}\longrightarrow\textswab{A}^{L}_{sur}$$
is a hull Lipschitz procedure. 
\end{prop}

Recall the definition of $\mathsf{A}$--factorable Lipschitz Operators in (\ref{Auto17}) and we assume here that $X=E$. It is evident $\textswab{A}^{L}:=Op^{L}(\mathsf{A})$ is a nonlinear operator.

\begin{prop}
Let $\mathsf{A}$ be a space ideal. Then $\left[Op^{L}(\mathsf{A})\right]_{sur}\subseteq Op^{L}(\mathsf{A^{sur}})$.
\end{prop}

\begin{proof}
Let $T\in Lip(E,F)$ belong to $\left[Op^{L}(\mathsf{A})\right]_{sur}$. Then $T Q_{E}\in Op^{L}(\mathsf{A})(E^{sur}, F)$, so $T Q_{E}=BA$ such that $A\in Lip(E^{sur},\mathbf{M})$, $B\in\mathfrak{L}(\mathbf{M}, F)$, and $\mathbf{M}\in\mathsf{A}$. Put $\mathbf{M_{0}}:=\overline{D_{B}}$, and let $B_{0}\in\mathfrak{L}(\mathbf{M_{0}}, F)$ be the bounded linear operator induced by $B$. Obviously $A ((x_{j})_{j})\in D_{B}$ for $(x_{j})_{j}\in D_{Q_{E}}$. Consequently $A(D_{Q_{E}})\subseteq\mathbf{M_{0}}$, and $A_{0}:=J_{\mathbf{M_{0}}}^{\mathbf{M}} A Q_{E}^{-1}$  is well--defined. Finally, $T=B_{0} A_{0}$ and $\mathbf{M_{0}}\in\mathsf{A^{sur}}$ imply $T\in Op^{L}(\mathsf{A^{sur}})$. This proves $\left[Op^{L}(\mathsf{A})\right]_{sur}\subseteq  Op^{L}(\mathsf{A^{sur}})$. \\
\end{proof}

\begin{rem}
A nonlinear ideal $\textswab{A}^{L}$ is called surjective if $\textswab{A}^{L}=\textswab{A}^{L}_{sur}$.
\end{rem}

The surjectivity of a nonlinear ideal $\textswab{A}^{L}$ means that it does not depend on the size of the source space $E$ whether or not a Lipschitz operator $T\in Lip(E,F)$  belongs to $\textswab{A}^{L}$.

\begin{prop}\label{Auto23}
A nonlinear ideal $\textswab{A}^{L}$ is surjective if and only if for every surjection $Q\in\mathfrak{L}(E, E_{0})$ and for every Lipschitz operator $T_{0}\in Lip(E_{0}, F)$ it follows from $T_{0} Q\in\textswab{A}^{L}(E, F)$ that $T_{0}\in\textswab{A}^{L}(E_{0}, F)$.
\end{prop}

\begin{proof}
To check the necessity we consider a surjective nonlinear ideal $\textswab{A}^{L}$. Let $T_{0}\in Lip(E_{0}, F)$ such that $T_{0} Q\in\textswab{A}^{L}(E, F)$. Since $E_{0}^{sur}$ has the lifting property, there is $B\in\mathfrak{L}(E_{0}^{sur}, E)$ with $Q_{E_{0}}=Q B$. Hence it follows from $T_{0} Q_{E_{0}} =\left(T_{0} Q\right) B\in\textswab{A}^{L}(E_{0}^{sur}, F)$ that $T_{0}\in\textswab{A}^{L}_{sur}(E_{0}, F)=\textswab{A}^{L}(E_{0}, F)$. 

Conversely, let us suppose that the given condition is satisfied. If $T_{0}\in\textswab{A}^{L}_{sur}(E_{0}, F)$, then $T_{0} Q_{E_{0}} \in\textswab{A}^{L}(E_{0}^{sur}, F)$. Since $Q_{E_{0}}$ is a surjection, we obtain $T_{0}\in\textswab{A}^{L}(E_{0}, F)$. Therefore $\textswab{A}^{L}=\textswab{A}^{L}_{sur}$, which proves the sufficiency. \\
\end{proof}

Recall the definitions of symmetric and injective nonlinear ideal in (\ref{Auto19}) and (\ref{Auto20}), respectively. We assume here that $X=E$. It is evident $\mathfrak{A}^{L}_{dual}$ and $\textswab{A}^{L}_{inj}$ is a nonlinear ideal.

\begin{lem}
 $\left(\mathfrak{A}_{dual}^{L}\right)_{inj}\subseteq\left(\mathfrak{A}_{sur}\right)^{L}_{dual}$.
\end{lem}

\begin{proof}
Let $T\in Lip(E,F)$ belong to $\left(\mathfrak{A}_{dual}^{L}\right)_{inj}$. Then $J_{F} T\in\mathfrak{A}_{dual}^{L}\left(E, F^{inj}\right)$ and $T^{\#}_{|_{\left(F^{inj}\right)^{*}}} J_{F}^{*}\in\mathfrak{A}\left({\left(F^{inj}\right)^{*}}, E^{\#}\right)$. Since $J_{F}^{*}$ is a surjection, it follows from \cite [Sec. 4.7.9] {P78} that $T^{\#}_{|_{\left(F^{inj}\right)^{*}}}\in\mathfrak{A}^{L}_{sur}$. Therefore $T\in\left(\mathfrak{A}_{sur}\right)^{L}_{dual}$. \\ 
\end{proof}

\begin{lem}
 For every surjective ideal the Lipschitz dual ideal is injective.
\end{lem} 

\begin{proof}
It follows from $\mathfrak{A}=\mathfrak{A}_{sur}$ that $\mathfrak{A}_{dual}^{L}=\left(\mathfrak{A}_{sur}\right)_{dual}^{L}\supseteq\left(\mathfrak{A}_{dual}^{L}\right)_{inj}$.  \\ 
\end{proof}

\begin{OP}
Every Banach space  $E_{sur}$  possesses the approximation property. Is it true that $\textswab{G}^{L}_{sur}=\textswab{R}^{L} ?$
\end{OP}
\subsection{Minimal Nonlinear Ideals}
Let $\mathfrak{A}$ be an ideal. A Lipschitz operator $T\in Lip(X,F)$ belongs to the minimal $\mathfrak{A}^{L}_{min}$ if $T=BT_{0}A$, where $B\in\mathfrak{G}(F_{0}, F)$, $T_{0}\in\mathfrak{A}(G_{0}, F_{0})$, and $A\in\textfrak{G}^{L}(X, G_{0})$. In the other words $\mathfrak{A}^{L}_{min}:=\mathfrak{G}\circ\mathfrak{A}\circ\textfrak{G}^{L}$.

\begin{prop}
$\mathfrak{A}^{L}_{min}$ is a nonlinear ideal. 
\end{prop}

\begin{proof}
The condition $\bf (\widetilde{NOI_0})$ satisfied. Since the elementary Lipschitz tensor $g\boxtimes e$ admits a factorization $$g\boxtimes e : X\stackrel{g\boxtimes 1}{\longrightarrow} \mathbb{K}\stackrel{1\otimes 1}{\longrightarrow}\mathbb{K}\stackrel{1\otimes e}{\longrightarrow} F,$$
where $1\otimes e\in\mathfrak{G}\left(\mathbb{K}, F\right)$, $1\otimes 1\in\mathfrak{A}\left(\mathbb{K}, \mathbb{K}\right)$, and $g\boxtimes 1\in\textfrak{G}^{L}\left(X, \mathbb{K}\right)$. To prove the condition $\bf (\widetilde{NOI_1})$, let $T_{i}\in\mathfrak{G}\circ\mathfrak{A}\circ\textfrak{G}^{L}(X, F)$. Then $T_{i}=B_{i}T_{0}^{i} A_{i}$, where $B_{i}\in\mathfrak{G}(F_{0}^{i}, F)$, $T_{0}^{i}\in\mathfrak{A}(G_{0}^{i}, F_{0}^{i})$, and $A_{i}\in\textfrak{G}^{L}(X, G_{0}^{i})$. Put $B:=B_{1}\circ Q_{1} + B_{2}\circ Q_{2}$, $T_{0}:=\tilde{J}_{1}\circ T_{0}^{1}\circ \tilde{Q}_{1} + \tilde{J}_{2}\circ T_{0}^{2}\circ \tilde{Q}_{2} $, and $A:=J_{1}\circ A_{1} + J_{2}\circ A_{2}$. Now $T_{1} + T_{2}= B\circ T_{0}\circ A$, $B\in\mathfrak{G}(F_{0}, F)$, $T_{0}\in\mathfrak{A}(G_{0}, F_{0})$, and $A\in\textfrak{G}^{L}(X, G_{0})$ imply $T_{1} + T_{2}\in\mathfrak{G}\circ\mathfrak{A}\circ\textfrak{G}^{L}(X, F)$. 

Let $A\in Lip(X_{0},X)$, $T\in\mathfrak{G}\circ\mathfrak{A}\circ\textfrak{G}^{L}(X, F)$, and $B\in\mathfrak{L}(F,R_{0})$. Then $T$ admits a factorization 
$$T: X\stackrel{\widetilde{A}}{\longrightarrow} G_{0}\stackrel{T_{0}}{\longrightarrow} F_{0}\stackrel{\widetilde{B}}{\longrightarrow} F,$$
where $\widetilde{B}\in\mathfrak{G}(F_{0}, F)$, $T_{0}\in\mathfrak{A}(G_{0}, F_{0})$, and $\widetilde{A}\in\textfrak{G}^{L}(X, G_{0})$. To show that $BTA\in\mathfrak{G}\circ\mathfrak{A}\circ\textfrak{G}^{L}(X_{0},R_{0})$. By using the non--linear composition ideal properties, we obtain $B\circ\widetilde{B}\in\mathfrak{G}\left(F_{0}, R_{0}\right)$ and $\widetilde{A}\circ A\in\textfrak{G}^{L}\left(X_{0}, G_{0}\right)$. Hence the Lipschitz operator $BTA$ admits a factorization 
$$BTA: X_{0}\stackrel{\widetilde{\widetilde{A\,}}}{\longrightarrow} G_{0}\stackrel{T_{0}}{\longrightarrow} F_{0}\stackrel{\widetilde{\widetilde{B\,}}}{\longrightarrow} R_{0},$$
where $\widetilde{\widetilde{B\,}}=B\circ\widetilde{B}$ and $\widetilde{\widetilde{A\,}}=\widetilde{A}\circ A$, hence the condition $\bf (\widetilde{NOI_2})$ satisfied. \\
\end{proof}

\begin{prop}
The rule $$min: \mathfrak{A}\longrightarrow\mathfrak{A}^{L}_{min}$$
is a monotone Lipschitz procedure. 
\end{prop}



\begin{rem}
\begin{itemize}
  \item It is evident $\mathfrak{A}^{L}_{min}\subseteq\textfrak{G}^{L}$.
	\item If $\textfrak{A}^{L}$ is a closed nonlinear ideal, then $\mathfrak{A}^{L}_{min}\subseteq\textfrak{A}^{L}$.
	\item A nonlinear ideal $\textfrak{A}^{L}$ is called minimal if $\mathfrak{A}^{L}_{min}\subseteq\textfrak{A}^{L}$.
\end{itemize}
\end{rem}

\begin{lem}
$\mathfrak{F}^{L}$ is a minimal nonlinear ideal. 
\end{lem}

\begin{proof}
Since $\mathfrak{F}^{L}$ is closed we obtain $\mathfrak{F}^{L}_{min}\subseteq\textfrak{F}^{L}$.
\end{proof}
\section{Lipschitz p--Normed Nonlinear Ideals}

Let $\textfrak{A}^{L}$ be a nonlinear ideal. A map $\mathbf{A}^{L}$ from $\textfrak{A}^{L}$ to $\mathbb{R}^{+}$ is called a \textbf{Lipschitz $p-$norm} $\left(0<p\leq 1\right)$ if the following conditions are satisfied:

\begin{description}
	\item[$\bf (\widetilde{QNOI_0})$] $\mathbf{A}^{L}\left(g\boxtimes e\right)=Lip(g)\cdot\left\|e\right\|$ for $g\in X^{\#}$ and $e\in F$.
	\item[$\bf (\widetilde{QNOI_1})$] The $p-$triangle inequality holds:
	$$\mathbf{A}^{L}\left(S + T\right)^{p}\leq\mathbf{A}^{L}(S)^{p} + \mathbf{A}^{L}(T)^{p} \ for \  S,\: T\in\textfrak{A}^{L}(X,F).$$
	\item[$\bf (\widetilde{QNOI_2})$] $\mathbf{A}^{L}\left(BTA\right)\leq\left\|B\right\|\mathbf{A}^{L}(T)\: Lip(A)$ for $A\in Lip(X_{0},X)$, $T\in\textfrak{A}^{L}(X,F)$, and $B\in\mathfrak{L}(F,F_{0})$.
\end{description}

\begin{rem}
\begin{itemize}
	\item A \textbf{Lipschitz $p$--Banach nonlinear ideal} $\left[\textfrak{A}^{L}, \mathbf{A}^{L}\right]$ is a nonlinear ideal $\textfrak{A}^{L}$ with a Lipschitz p--norm $\mathbf{A}^{L}$ such that all linear spaces $\textfrak{A}^{L}(X,F)$ are complete. 

 \item We call $\mathbf{A}^{L}$ a \textbf{strongly nonlinear ideal norm} if the condition $\bf (\widetilde{QNOI_0})$ is replaced by:
  $$\mathbf{A}^{L}\left(g\boxtimes e\right)\leq Lip(g)\cdot\left\|e\right\|\ \ \  for\ \  g\in X^{\#}\ \  and\ \  e\in F.$$
\end{itemize}
\end{rem}

\begin{prop}
Let $\left[\textfrak{A}^{L}, \mathbf{A}^{L}\right]$ be a Lipschitz $p$--normed nonlinear ideal. Then $Lip(T)\leq\mathbf{A}^{L}(T)$ for all $T\in\textfrak{A}^{L}$.  
\end{prop}

\begin{proof}
Let $T$ be an arbitrary Lipschitz operator in $\textfrak{A}^{L}(X,F)$. 
\begin{align}
Lip(T)=\left\|T^{\#}_{|_{{F}^{*}}}\right\|=\sup\left\{Lip(T^{\#} b^{*}) : b^{*}\in B_{{F}^{*}}\right\}&=\sup\left\{Lip\left(b^{*}\circ T\right) : b^{*}\in B_{{F}^{*}}\right\}\nonumber \\
&=\sup\left\{\mathbf{A}^{L}\left((b^{*}\circ T) \boxtimes 1\right) : b^{*}\in B_{{F}^{*}}\right\} \nonumber \\
&=\sup\left\{\mathbf{A}^{L}\left(b^{*}\circ T\right)  : b^{*}\in B_{{F}^{*}}\right\} \nonumber \\
&\leq\mathbf{A}^{L}(T). \nonumber 
\end{align}
\end{proof}


\begin{rem}
If $p=1$, then $\mathbf{A}^{L}$ is simply called a \textbf{Lipschitz norm} and $\left[\textfrak{A}^{L}, \mathbf{A}^{L}\right]$ is said to be a \textbf{Lipschitz Banach nonlinear ideal}.  
\end{rem}

We now formulate an important criterion which will be permanently used in the sequel.

\begin{prop}\label{Auto12}
Let $\textfrak{A}^{L}$ be a subclass of $Lip$ with an $\mathbb{R}^{+}-$valued function $\mathbf{A}^{L}$ such that the following conditions are satisfied $\left(0<p\leq 1\right)$:

\begin{description}
	\item[$\bf (1)$] $g\boxtimes e\in\textfrak{A}^{L}(X,F)$ and $\mathbf{A}^{L}\left(g\boxtimes e\right)\leq Lip(g)\cdot\left\|e\right\|$ for $g\in X^{\#}$ and $e\in F$.
	\item[$\bf (2)$] It follows from $S_{1}, S_{2}, S_{3},\ldots\in\textfrak{A}^{L}(X,F)$ and $\sum\limits_{n=1}^{\infty}\mathbf{A}^{L}(S_{n})^{p}<\infty$ that $S:=\sum\limits_{n=1}^{\infty} S_{n}\in\textfrak{A}^{L}(X,F)$ and $\mathbf{A}^{L}(S)^{p}\leq\sum\limits_{n=1}^{\infty}\mathbf{A}^{L}(S_{n})^{p}$.
	\item[$\bf (3)$] $A\in Lip(X_{0},X)$, $T\in\textfrak{A}^{L}(X,F)$, and $B\in\mathfrak{L}(F,F_{0})$ imply $BTA\in\textfrak{A}^{L}(X_{0},F_{0})$ and $\mathbf{A}^{L}\left(BTA\right)\leq\left\|B\right\|\mathbf{A}^{L}(T)\: Lip(A)$.
\end{description}
Then $\left[\textfrak{A}^{L}, \mathbf{A}^{L}\right]$ is a strongly Lipschitz $p-$Banach nonlinear ideal.
\end{prop}

\begin{proof}
The only point is to observe that $\bf (2)$ summarizes the $p-$triangle inequality and the completeness, as well. \\
\end{proof}


\subsection{Lipschitz Nuclear Operators}

A Lipschitz operator $T\in Lip(X,F)$ is called Lipschitz nuclear if $$T=\sum\limits_{j=1}^{\infty} g_{j}\boxtimes e_{j},$$ with $g_{1}, g_{2}, g_{3}, \ldots\in X^{\#}$ and $e_{1}, e_{2}, e_{3}, \ldots\in F$ such that $\sum\limits_{j=1}^{\infty} Lip(g_{j})\left\|e_{j}\right\|<\infty$.

We put $$\mathbf{N}^{L}(T):=\inf\sum\limits_{j=1}^{\infty} Lip(g_{j})\left\|e_{j}\right\|,$$ where the infimum is taken over all so--called Lipschitz nuclear representations described above. 

The class of all Lipschitz nuclear operators is denoted by $\textfrak{N}^{L}$.

\begin{rem}
The series $\sum\limits_{j=1}^{\infty} g_{j}\boxtimes e_{j}$ converges in the Lipschitz norm topology of $Lip(X,F)$. Therefore Lipschitz nuclear operators can be approximated by Lipschitz finite rank operators.
\end{rem}

\begin{prop}
$\left[\textfrak{N}^{L}, \mathbf{N}^{L}\right]$ is a strongly Lipschitz normed nonlinear ideal. 
\end{prop}

\begin{proof}
We use criterion \ref{Auto12}.
\begin{description}
	\item[$\bf (1)$] Let $g\in X^{\#}$ and $e\in F$, it follows that $g\boxtimes e\in\textfrak{N}^{L}(X,F)$ and $\mathbf{N}^{L}(g\boxtimes e)\leq Lip(g)\cdot\left\|e\right\|$.
	\item[$\bf (2)$] Let $T_{1}, T_{2}, T_{3},\ldots\in\textfrak{N}^{L}(X,F)$ such that $\sum\limits_{n=1}^{\infty}\mathbf{N}^{L}(T_{n})<\infty$.	Given $\epsilon > 0$, choose nuclear representations $T_{n}=\sum\limits_{j=1}^{\infty} g_{nj}\boxtimes e_{nj}$ with $\sum\limits_{j=1}^{\infty} Lip(g_{nj})\left\|e_{nj}\right\|\leq (1+\epsilon)\; \mathbf{N}^{L}(T_{n})$. Then $T:=\sum\limits_{n=1}^{\infty} T_{n}=\sum\limits_{n, j=1}^{\infty} g_{nj}\boxtimes e_{nj}$ and $\sum\limits_{n, j=1}^{\infty} Lip(g_{nj})\left\|e_{nj}\right\|\leq (1+\epsilon) \sum\limits_{n=1}^{\infty}\mathbf{N}^{L}(T_{n})$
imply $T\in\textfrak{N}^{L}(X,F)$	and $\mathbf{N}^{L}(T)\leq (1+\epsilon) \sum\limits_{n=1}^{\infty}\mathbf{N}^{L}(T_{n})$.
	
\item[$\bf (3)$] Let $T\in\textfrak{N}^{L}(X,F)$ and $\epsilon > 0$. Consider a nuclear representation $T=\sum\limits_{j=1}^{\infty} g_{j}\boxtimes e_{j}$ such that $\sum\limits_{j=1}^{\infty} Lip(g_{j})\left\|e_{j}\right\|\leq (1+\epsilon)\; \mathbf{N}^{L}(T)$. If $A\in Lip(X_{0},X)$ and $B\in\mathfrak{L}(F,F_{0})$, then $BTA=\sum\limits_{j=1}^{\infty} T^{\#} g_{j}\boxtimes B e_{j}$ and $\sum\limits_{j=1}^{\infty} Lip(T^{\#} g_{j}) \left\|B e_{j}\right\|\leq (1+\epsilon)\; \left\|B\right\|\mathbf{N}^{L}(T)\: Lip(A)$. This proves that $BTA\in\textfrak{N}^{L}(X_{0},F_{0})$ and $\mathbf{N}^{L}\left(BTA\right)\leq (1+\epsilon)\left\|B\right\|\mathbf{N}^{L}(T)\: Lip(A)$.
\end{description}
\end{proof}



\subsection{Lipschitz Hilbert Operators}

A Lipschitz operator $T\in Lip(X,F)$ is called Lipschitz Hilbert operator if $T=BA$ with $A\in Lip(X,H)$ and $B\in\mathfrak{L}(H, F)$, where $H$ is a Hilbert space. 

We put $$\mathbf{H}^{L}(T):=\inf \left\|B\right\| Lip(A),$$ where the infimum is taken over all possible factorizations. 

The class of all Lipschitz Hilbert operators is denoted by $\textfrak{H}^{L}$.

\begin{prop}
$\left[\textfrak{H}^{L}, \mathbf{H}^{L}\right]$ is a strongly Lipschitz normed nonlinear ideal. 
\end{prop}

\begin{proof}
We use criterion \ref{Auto12}.
\begin{description}
	\item[$\bf (1)$] Let $g\in X^{\#}$ and $e\in F$, since $\mathbb{K}$ is a Hilbert space, we have $g\boxtimes e =(1\otimes e)\circ (g\boxtimes 1)\in\textfrak{H}^{L}{(X, F)}$, where $1\otimes e\in\mathfrak{L}(\mathbb{K},F)$ and $g\boxtimes 1\in Lip(X,\mathbb{K})$ and $\mathbf{H}^{L}(g\boxtimes e)\leq Lip(g)\cdot\left\|e\right\|$.
	\item[$\bf (2)$] Let $T_{1}, T_{2}, T_{3},\ldots\in\textfrak{H}^{L}(X,F)$ such that $\sum\limits_{n=1}^{\infty}\mathbf{H}^{L}(T_{n})<\infty$.	Given $\epsilon > 0$, we choose factorizations $T_{n}=B_{n}A_{n}$ such that $A_{n}\in Lip(X,H_{n})$ and $B_{n}\in\mathfrak{L}(H_{n}, F)$   satisfy the conditions $$Lip(A_{n})^{2}\leq (1+\epsilon) \:\mathbf{H}^{L}(T_{n}) \ \ \ and \ \ \ \left\|A_{n}\right\|^{2}\leq (1+\epsilon) \:\mathbf{H}^{L}(T_{n}).$$ From the Hilbert space $H:=\ell_{2}(H_{n})$. Put $$A:=\sum\limits_{n=1}^{\infty} J_{n} A_{n}\ \ \  and \ \ \  B:=\sum\limits_{n=1}^{\infty} B_{n} Q_{n}.$$ Then $$Lip(A)^{2}\leq\sum\limits_{n=1}^{\infty} Lip(A_{n})^{2}\ \ \  and \ \ \  \left\|B\right\|^{2}\leq\sum\limits_{n=1}^{\infty} \left\|B_{n}\right\|^{2}.$$ Finally, it follows from $T=BA$ that $T\in\textfrak{H}^{L}(X,F)$. Moreover, we have $$\mathbf{H}^{L}(T)\leq\left\|B\right\|\  Lip(A)\leq (1+\epsilon)\sum\limits_{n=1}^{\infty}\mathbf{H}^{L}(T_{n}).$$ 
	
\item[$\bf (3)$] This property is trivial.
\end{description}
\end{proof}

\subsection{Products of Lipschitz $r$--Normed Nonlinear Ideals}

Let $\left[\mathfrak{A}, \mathbf{A}\right]$ and $\left[\textfrak{A}^{L}, \mathbf{A}^{L}\right]$ be $p$--normed ideal and Lipschitz $q$--normed nonlinear ideal, respectively. For every Lipschitz operator $T\in Lip(X,F)$ belonging to the product $\mathfrak{A}\circ\textfrak{A}^{L}$ we put $$\mathbf{A}\circ\mathbf{A}^{L}(T):=\inf \mathbf{A}(B)\mathbf{A}^{L}(A),$$
where the infimum is taken over all factorizations $T=B\circ A$ with $B\in\mathfrak{A}\left(M, F\right)$ and $A\in\textfrak{A}^{L}\left(X, M\right)$ and $\frac{1}{r}:= \frac{1}{p} + \frac{1}{q}$.

\begin{rem}
The product $\left[\mathfrak{A}\circ\textfrak{A}^{L}, \mathbf{A}\circ\mathbf{A}^{L}\right]$ will frequently be written as $\left[\mathfrak{A}, \mathbf{A}\right]\circ\left[\textfrak{A}^{L}, \mathbf{A}^{L}\right]$. 
\end{rem} 

\begin{prop}
$\left[\mathfrak{A}\circ\textfrak{A}^{L}, \mathbf{A}\circ\mathbf{A}^{L}\right]$ is strongly Lipschitz $r$--Banach nonlinear ideal. 
\end{prop}

\begin{proof}
From Proposition \ref{Auto8} we have $\mathfrak{A}\circ\textfrak{A}^{L}$ is nonlinear ideal. The condition $\bf (\widetilde{QNOI_0})$ satisfied. To prove condition $\bf (\widetilde{QNOI_1})$, let $T_{1}$ and $T_{2}$ in $\mathfrak{A}\circ\textfrak{A}^{L}(X, F)$. Given $\epsilon > 0$. Then there exists Banach spaces $M_{1}$, $M_{2}$ and operators $B_{i}\in\mathfrak{A}\left(M_{i}, F\right)$ and $A_{i}\in\textfrak{A}^{L}\left(X, M_{i}\right)$ so that $T_{i}=B_{i}\circ A_{i}$ with

$$\mathbf{A}(B_{1})\leq\left[(1+\epsilon)\cdot\mathbf{A}\circ\mathbf{A}^{L}(T_{1})\right]^{\frac{r}{p}},$$

$$\mathbf{A}(B_{2})\leq\left[(1+\epsilon)\cdot\mathbf{A}\circ\mathbf{A}^{L}(T_{2})\right]^{\frac{r}{p}},$$

$$\mathbf{A}^{L}(A_{1})\leq\left[(1+\epsilon)\cdot\mathbf{A}\circ\mathbf{A}^{L}(T_{1})\right]^{\frac{r}{q}},$$

$$\mathbf{A}^{L}(A_{2})\leq\left[(1+\epsilon)\cdot\mathbf{A}\circ\mathbf{A}^{L}(T_{2})\right]^{\frac{r}{q}}.$$

Therefore the construction of condition $\bf (\widetilde{NOI_1})$ in Proposition \ref{Auto8} implies that 
\begin{align}
\mathbf{A}\circ\mathbf{A}^{L}\left(T_{1} + T_{2}\right)^{r}&=\mathbf{A}\circ\mathbf{A}^{L}\left(B\circ A\right)^{r}\leq\mathbf{A}(B)^{r}\mathbf{A}^{L}(A)^{r}                                      \nonumber \\
&\leq\left[\mathbf{A}(B_{1})^{p} + \mathbf{A}(B_{2})^{p}\right]^{\frac{r}{p}}\left[\mathbf{A}^{L}(A_{1})^{q} + \mathbf{A}^{L}(A_{2})^{q}\right]^{\frac{r}{q}}                                       \nonumber \\
&\leq\left[(1+\epsilon)^{r}\left(\mathbf{A}\circ\mathbf{A}^{L}(T_{1})^{r} + \mathbf{A}\circ\mathbf{A}^{L}(T_{2})^{r}\right)\right]^{\frac{r}{p}}\left[(1+\epsilon)^{r}\left(\mathbf{A}\circ\mathbf{A}^{L}(T_{1})^{r} + \mathbf{A}\circ\mathbf{A}^{L}(T_{2})^{r}\right)\right]^{\frac{r}{q}}                                                       \nonumber \\
&\leq (1+\epsilon)^{r}\left(\mathbf{A}\circ\mathbf{A}^{L}(T_{1})^{r} + \mathbf{A}\circ\mathbf{A}^{L}(T_{2})^{r}\right).   \nonumber 
\end{align}
To prove condition $\bf (\widetilde{QNOI_2})$, let $A\in Lip(X_{0},X)$, $T\in\mathfrak{A}\circ\textfrak{A}^{L}(X,F)$, and $B\in\mathfrak{L}(F,F_{0})$. From the construction of condition $\bf (\widetilde{NOI_2})$ in Proposition \ref{Auto8} we have
\begin{align}
\mathbf{A}\circ\mathbf{A}^{L}(BTA)=\mathbf{A}\circ\mathbf{A}^{L}(\widetilde{\widetilde{B\,}}\circ\widetilde{\widetilde{A\,}})&:=\inf\mathbf{A}(B\circ\widetilde{B})\mathbf{A}^{L}(\widetilde{A}\circ A) \nonumber \\
&\leq\mathbf{A}(B\circ\widetilde{B})\mathbf{A}^{L}(\widetilde{A}\circ A)  \nonumber \\
&\leq\left\|B\right\| \: \mathbf{A}(\widetilde{B})\mathbf{A}^{L}(\widetilde{A}) \: Lip(A)   \nonumber \\
&\leq \left\|B\right\| \: \mathbf{A}\circ\mathbf{A}^{L}(T) \: Lip(A).   \nonumber 
\end{align}
 To show the completeness, let $T_{n}\in\mathfrak{A}\circ\textfrak{A}^{L}(X, F)$ such that $\sum\limits_{n=1}^{\infty}\mathbf{A}\circ\mathbf{A}^{L}(T_{n})^{r}<\infty$. Given $\epsilon > 0$, we can find factorizations $T=B_{n}\circ A_{n}$ such that the following conditions are satisfied:
$$B_{n}\in\mathfrak{A}\left(M_{n}, F\right) \ \ \  and \ \ \  \mathbf{A}(B_{n})\leq\left[(1+\epsilon)\:\mathbf{A}\circ\mathbf{A}^{L}(T_{n})\right]^{\frac{r}{p}},$$

$$A_{n}\in\textfrak{A}^{L}\left(X , M_{n}\right) \ \ \  and \ \ \  \mathbf{A}^{L}(A_{n})\leq\left[(1+\epsilon)\:\mathbf{A}\circ\mathbf{A}^{L}(T_{n})\right]^{\frac{r}{q}}.$$

Put $B:=\sum\limits_{n=1}^{\infty} B_{n} Q_{n}$, $A:=\sum\limits_{n=1}^{\infty} J_{n} A_{n}$, and $M:=\ell_{2}(M_{n})$. Then

$$\sum\limits_{n=1}^{\infty}\mathbf{A}(B_{n})^{p}\leq (1+\epsilon)^{r}\:\sum\limits_{n=1}^{\infty}\mathbf{A}\circ\mathbf{A}^{L}(T_{n})^{r},$$
and
$$\sum\limits_{n=1}^{\infty}\mathbf{A}^{L}(A_{n})^{q}\leq (1+\epsilon)^{r}\: \sum\limits_{n=1}^{\infty}\mathbf{A}\circ\mathbf{A}^{L}(T_{n})^{r}.$$
imply $B\in\mathfrak{A}\left(M, F\right)$ and $A\in\textfrak{A}^{L}\left(X , M\right)$. Moreover, we have 

$$\mathbf{A}(B)^{p}\leq (1+\epsilon)^{r}\:\sum\limits_{n=1}^{\infty}\mathbf{A}\circ\mathbf{A}^{L}(T_{n})^{r},$$
and
$$\mathbf{A}^{L}(A)^{q}\leq (1+\epsilon)^{r}\: \sum\limits_{n=1}^{\infty}\mathbf{A}\circ\mathbf{A}^{L}(T_{n})^{r}.$$

Since $T:=\sum\limits_{n=1}^{\infty} T_{n}$ has the factorization $T=B\circ A$ it follows that $T\in\mathfrak{A}\circ\textfrak{A}^{L}(X, F)$ and 

$$\mathbf{A}\circ\mathbf{A}^{L}(T)^{r}\leq (1+\epsilon)^{r}\: \sum\limits_{n=1}^{\infty}\mathbf{A}\circ\mathbf{A}^{L}(T_{n})^{r}.$$
Hence Proposition \ref{Auto12} yields the assertion. \\
\end{proof}

\subsection{Quotients of Lipschitz $r$--Normed Nonlinear Ideals}

Let $\left[\mathfrak{A}, \mathbf{A}\right]$ and $\left[\textfrak{A}^{L}, \mathbf{A}^{L}\right]$ be $p$--normed ideal and Lipschitz $q$--normed nonlinear ideal, respectively. For every Lipschitz operator $T\in Lip(X,F)$ belonging to the left--hand quotient $\mathfrak{A}^{-1}\circ\textfrak{A}^{L}$ we put

$$\mathbf{A^{-1}}\circ\mathbf{A}^{L}(T):=\sup\left\{\mathbf{A}^{L}(B\circ T) : B\in\mathfrak{A}(F,F_{0}), \mathbf{A}(B)\leq 1\right\},$$

where $F_{0}$ is an arbitrary Banach space and $\frac{1}{r}:= \frac{1}{p} + \frac{1}{q}$. 

\begin{rem}
The quotient  $\left[\mathfrak{A^{-1}}\circ\textfrak{A}^{L}, \mathbf{A^{-1}}\circ\mathbf{A}^{L}\right]$ will frequently be written as ${\left[\mathfrak{A}, \mathbf{A}\right]}^{-1}\circ\left[\textfrak{A}^{L}, \mathbf{A}^{L}\right]$. 
\end{rem} 

\begin{prop}
$\left[\mathfrak{A^{-1}}\circ\textfrak{A}^{L}, \mathbf{A^{-1}}\circ\mathbf{A}^{L}\right]$ is strongly Lipschitz $r$--Banach nonlinear ideal. 
\end{prop}

\begin{proof}
The main point is to establish the existence of $\mathbf{A^{-1}}\circ\mathbf{A}^{L}$. Therefore let us suppose that the supremum is not finite for some Lipschitz operator $T\in\mathfrak{A^{-1}}\circ\textfrak{A}^{L}(X, F)$. Then we can find $B_{n}\in\mathfrak{A}(F,F_{n})$ such that $$\mathbf{A}(B_{n})\leq (2)^{-n}\ \ \  and \ \ \ \mathbf{A}^{L}(B_{n}\circ T)\geq n \ \ for \ \ \  n=1,2,3,\ldots$$

Put $F_{0}:=\ell_{2}(F_{n})$. Since $$\mathbf{A}\left(\sum\limits_{n=h+1}^{k} J_{n} B_{n}\right)\leq\sum\limits_{j=1}^{\infty} \mathbf{A}(B_{h+j})\leq (2)^{-h},$$ 
the partial sums $\sum\limits_{n=1}^{k} J_{n} B_{n}$ form an $\mathbf{A}-$Cauchy sequence. Consequently $B:=\sum\limits_{n=1}^{\infty} J_{n} B_{n}$ belongs to $\mathfrak{A}$, and we obtain $n\leq\mathbf{A}^{L}(B_{n}\circ T)=\mathbf{A}^{L}(Q_{n} B T)\leq \mathbf{A}^{L}(B T)$, which is a contradiction. 

Finally, it is easy to check  the  nonlinear ideal properties and the completeness, as well. \\
\end{proof}



\end{document}